%% file: main.tex
\documentclass[a4paper,fleqn]{cls/cas-dc}

\usepackage[numbers]{natbib}

\usepackage{amsthm,dsfont,mathtools}
\usepackage{nicematrix}
\usepackage{array,tabularx,adjustbox,multirow,booktabs}
\usepackage{textcomp}
\usepackage{stfloats}
\usepackage{url}
\usepackage{verbatim}
\usepackage{algorithmic}
\usepackage[ruled,vlined]{algorithm2e}
\usepackage[font=footnotesize]{caption}
\usepackage[font=footnotesize]{subcaption}
\usepackage{todonotes}
\usepackage{stackengine}
\usepackage{dirtytalk}

\DeclarePairedDelimiter\ceil{\lceil}{\rceil}

\allowdisplaybreaks

\DeclareMathOperator{\Quant}{Quantile}
\DeclareMathOperator{\dist}{dist}

\newtheorem{theorem}{Theorem}[section]

\newtheorem{lemma}[theorem]{Lemma}
\newtheorem{Assumption}[theorem]{Assumption}

\newtheorem{remark}[theorem]{Remark}
\theoremstyle{definition}
\newtheorem{definition}[theorem]{Definition}
\newtheorem{problem}{Problem}
\newtheorem*{problem*}{Problem}
\newtheorem{subproblem}{Problem}[problem]

\newcommand\xqed[1]{%
  \leavevmode\unskip\penalty9999 \hbox{}\nobreak\hfill
  \quad\hbox{#1}}

\newcommand\resultend{\xqed{$\bullet$}}

\newcommand{\longthmtitle}[1]{\mbox{}\textup{\textbf{(#1):}}}

\input{cls/sym.tex}

\begin{document}
\let\WriteBookmarks\relax
\def\floatpagepagefraction{1}
\def\textpagefraction{.001}

% Short title / running header
\shorttitle{A Unified Perspective on CP and Wasserstein DRO for Uncertainty Quantification}
\shortauthors{K. Long et al.}

% Main title
\title[mode=title]{A Unified Perspective on Conformal Prediction and Wasserstein Distributionally Robust Optimization for Uncertainty Quantification}

\author[1]{Kehan Long}[%
  % orcid=0000-0000-0000-0000,   % TODO: fill in if available
  ]
% \cormark[1]
\author[2]{Yiqi Zhao}
\author[3]{Pol Mestres}
\author[2,4]{Lars Lindemann}
\author[1]{Nikolay Atanasov}
\author[1]{Jorge Cort\'es}

\affiliation[1]{organization={Contextual Robotics Institute, University of California San Diego},
  city={La Jolla}, state={CA}, country={USA}}
  \affiliation[2]{organization={Thomas Lord Department of Computer Science, University of Southern California},
  city={Los Angeles}, state={CA}, country={USA}}
\affiliation[3]{organization={Department of Mechanical and Civil Engineering, California Institute of Technology},
  city={Pasadena}, state={CA}, country={USA}}
\affiliation[4]{organization={Automatic Control Laboratory, ETH Zürich},
  city={Zürich}, country={Switzerland}}

% \cortext[1]{Corresponding author.
%   Email: k3long@ucsd.edu}

\begin{abstract}
Uncertainty quantification from finite data is central to machine learning, optimization, and automation systems, where decisions must remain reliable under limited samples and test-time distribution shift. Conformal prediction (CP) and distributionally robust optimization (DRO) offer two complementary approaches: CP constructs data-dependent prediction sets with distribution-free finite-sample validity under exchangeability, while DRO optimizes worst-case performance over an ambiguity set around an empirical distribution. We develop a unified probabilistic perspective on CP and DRO by viewing both as ways to turn finite calibration data into a data-dependent quantile estimator that a test score falls below with high probability. From this perspective, CP and DRO correct the empirical quantile along two coordinates of the same family of estimators: CP inflates the quantile level, whereas DRO shifts the quantile value through an ambiguity radius. Both methods provide the same calibration-conditional guarantee for the true distribution, requiring the target coverage to hold with high probability over the calibration sample. Their constructions differ, however: CP uses a closed-form, distribution-free level correction, while DRO uses a value-space correction whose certified radius depends on properties of the unknown distribution and additionally guarantees coverage uniformly over the ambiguity set. This distinction emerges in the tails of the score distribution. Because CP relies on sparse upper-tail order statistics of the calibration samples, its level inflation barely moves the estimator when those samples are dense near the target quantile but overshoots when they are sparse, whereas a well-chosen DRO radius corrects in value space and may avoid this overshoot. We further extend the comparison to test-time distribution shift and derive CP- and DRO-based estimators under the L\'evy--Prokhorov model that satisfy the same two-level guarantee. We validate these findings across image classification, multiple-choice question answering, and autonomous-driving trajectory prediction.
\end{abstract}

\begin{keywords}
Uncertainty quantification \sep Distribution shift \sep Conformal prediction \sep Distributionally robust optimization  
\end{keywords}

\maketitle

\input{tex/Intro.tex}

\input{tex/Related_Work.tex}

\input{tex/ProblemFormulation.tex}

% \input{tex/Optimization.tex}
% \input{tex/Composition_Functions}
\input{tex/Conclusion.tex}

%==================================================================%
% References
\bibliographystyle{cas-model2-names}
\bibliography{ref}

\end{document}

%% file: cls/sym.tex
\newcommand{\calM}{{\cal M}}

\newcommand{\calX}{{\cal X}}

\newcommand{\bbP}{\mathbb{P}}
\newcommand{\bbQ}{\mathbb{Q}}

%% file: tex/Intro.tex
\section{Introduction}
\label{sec: intro}

Modern applications in machine learning, robotics, and decision-making increasingly rely on data-driven methods. A fundamental challenge in these settings is to \emph{quantify uncertainty using a finite set of samples}. Conformal prediction (CP) \cite{shafer2008tutorial, angelopoulos2021gentle} addresses this by constructing prediction sets with finite-sample distribution-free coverage guarantees under mild assumptions, such as data exchangeability. Distributionally robust optimization (DRO) \cite{Esfahani2018DatadrivenDR, kuhn2025distributionally} instead optimizes worst-case performance over an ambiguity set of distributions, centered at the empirical distribution, yielding decisions that remain robust under distributional shift. CP and DRO originate from different communities and employ different problem formulations, leading to distinct methodologies, guarantees, and trade-offs. Nevertheless, both approaches can provide finite-sample high-coverage statistical guarantees from i.i.d. data. 

Consider a set of calibration samples and a test sample drawn from the same unknown distribution, each assigned a scalar score. The goal is to construct a threshold from the calibration scores such that the test score falls below it with a target probability. Under \emph{marginal} validity, this probability is taken jointly over the calibration and test samples. The target coverage then holds on average over possible calibration sets, but need not hold for the particular set observed. In contrast, \emph{calibration-conditional} validity requires the target coverage to hold for the calibration set actually drawn. Because a finite calibration sample does not fully characterize the tail of the underlying distribution, this cannot hold for every calibration set, and one instead requires it to hold for all but a small fraction of them. We call the resulting statement a \emph{two-level} guarantee: one level is the target coverage for the test sample, and the other is the confidence with which that coverage holds over the random calibration set.

This paper presents a comparative study of CP and DRO in the context of uncertainty quantification. We summarize their theoretical formulations, statistical properties, finite-sample behaviors, and provide side-by-side comparisons of coverage guarantees, conservativeness, and numerical performance. Our \textbf{contributions} are summarized as follows.
\begin{itemize}
    \item We develop a unified probabilistic perspective on CP and Wasserstein DRO by formulating both as finite-sample data-dependent quantile-estimation methods with calibration-conditional coverage guarantees.

    \item In the absence of test-time distribution shift, we connect CP and DRO through data-dependent quantile estimation: CP inflates the empirical quantile level, whereas DRO adds a value-space ambiguity radius. Both provide the same calibration-conditional guarantee for the true distribution, while DRO additionally certifies coverage uniformly over the realized ambiguity set. We further characterize their finite-sample and asymptotic behavior and explain how the score density near the target quantile affects conservativeness.

    \item Under test-time distribution shift, we compare CP and DRO using shift models based on Wasserstein and L\'evy--Prokhorov distances. We derive both CP and DRO estimators with calibration-conditional two-level guarantees and show how value-space and quantile-level corrections account for different forms of distribution shift.

    \item We illustrate our theoretical findings on three representative tasks: image classification, multiple-choice question answering, and autonomous-driving trajectory prediction. The experiments compare the empirical coverage and conservativeness of CP and DRO estimators, assess satisfaction of the calibration-conditional two-level guarantees, and evaluate their robustness under distribution shift.
\end{itemize}

The paper aims to provide theoretical basis and practical guidance for selecting uncertainty-quantification methods in machine learning and decision-making from finite data.

%% file: tex/Related_Work.tex
\section{Related Work}
\label{sec: related}

We review CP, DRO, and related approaches to uncertainty quantification, with an emphasis on applications in machine learning, control, and robotics.

\subsubsection*{Conformal prediction}

% wraps a predictive model to 
CP addresses uncertainty quantification when the data-generating distribution is unknown. Instead of relying on parametric assumptions, CP constructs prediction sets or intervals that achieve finite-sample coverage guarantees under the mild assumption of data exchangeability \cite{shafer2008tutorial,angelopoulos2021gentle}. The idea originates from \cite{vovk2005algorithmic} and was reformulated for regression problems by \cite{lei2018distribution}, whose split-conformal construction underlies most modern use. Beyond this marginal-validity perspective, Vovk \cite{vovk2012conditional} studied several notions of conditional validity for conformal predictors. Two are particularly relevant here. The first conditions on the training data, referred to as \emph{training-conditional validity}, and corresponds to what we call calibration-conditional validity in this paper; Bian and Barber \cite{bian2023training} characterize when such guarantees are attainable. The second conditions on the test input, for which exact distribution-free guarantees are generally unattainable. Requiring coverage conditional on arbitrary subsets of the input space can lead to highly conservative prediction sets, whereas meaningful guarantees can be recovered over suitably restricted classes of conditioning sets \cite{foygel2021limits}. Conformalized quantile regression \cite{romano2019conformalized} seeks to improve conditional adaptivity in practice by producing intervals whose widths adapt to input-dependent noise, or heteroscedasticity, while retaining marginal coverage.
The results above assume exchangeability between calibration and test data. A separate line of work relaxes this assumption, including methods for covariate shift with a known likelihood ratio \cite{tibshirani2019covariate} and bounds on the coverage gap under departures from exchangeability \cite{barber2023beyond}. Angelopoulos et al. \cite{angelopoulos2024theoretical} provide a unified theoretical treatment of conformal prediction that also covers distribution-shift settings. Several seminal works \cite{singh2024uncertainty,zhou2025conformal,shorinwa2025survey, karimi2023quantifying, lindemann2025formal, strawn2023conformal} have established CP as a popular tool for reliable uncertainty quantification in machine learning, control, and autonomous systems.

In machine learning, CP offers a model-agnostic approach for converting model outputs into finite-sample calibrated prediction sets. Angelopoulos et al.~\cite{angelopoulos2020uncertainty} developed CP-based uncertainty sets for image classification that remain valid while being substantially more compact than standard calibration baselines, and \cite{lu2023federated} extended CP to federated learning by introducing a weaker notion of partial exchangeability suited to heterogeneous clients. In high-stakes applications, such as clinical imaging, \cite{vazquez2022conformal,lu2022fair} highlighted the growing role of subgroup-adaptive CP coverage for fairer uncertainty quantification. More recently, CP has been adapted to large language models for multiple-choice question answering \cite{kumar2023conformal}, generative language modeling \cite{quach2023conformal}, and improved validity guarantees for LLM outputs through enhanced conformal inference procedures \cite{cherian2024large}.

CP has also been increasingly adopted in control, where it converts model and prediction uncertainty into finite-sample certificates for constraint satisfaction, safety, and verification. It has been used for distribution-free optimal control of linear stochastic systems \cite{vlahakis2024conformal} and to quantify and robustify the uncertainty of model-based controllers \cite{chee2024uncertainty}. CP has further been embedded in stochastic model predictive control \cite{fernandezzapico2025stochastic} and in perception-based control under sensor uncertainty \cite{yang2023safe_perception}. Furthermore, it has been paired with high-level specifications and verification, including signal temporal logic control \cite{yu2026signal}, conformal predictive programming for chance-constrained optimization \cite{zhao2024conformal_cc_opti}, safety filters for reinforcement learning \cite{strawn2023conformal}, and runtime verification of autonomous systems \cite{lindemann2025formal, zhao2024robust_cp_verification}.

CP has also been applied in autonomous systems and robotics to convert model predictions into finite-sample calibrated uncertainty sets for planning, safety assurance, and interaction. Lindemann et al.~\cite{lindemann2023safe} incorporated CP into MPC-based safe planning in dynamic environments by calibrating trajectory-prediction uncertainty, while \cite{sun2023conformal} used CP to quantify uncertainty in diffusion dynamics models for uncertainty-aware planning. Luo et al.~\cite{luo2024sample} leveraged conformal calibration to obtain sample-efficient safety assurances with guaranteed false-negative rates, and \cite{lekeufack2024conformal} extended this perspective from set prediction to direct calibration of autonomous decisions. Seo et al.~\cite{seo2025uncertainty} used CP to calibrate epistemic-uncertainty thresholds in latent safety filters for avoiding out-of-distribution failures in robot navigation and manipulation. In human-robot interaction, \cite{lidard2024risk} used set-valued intent prediction to achieve risk-calibrated interaction under uncertain human intent.

\subsubsection*{Distributionally robust optimization}

DRO addresses optimization under uncertainty by assuming that the true distribution lies within a specified family of distributions, known as an \emph{ambiguity set}. This is a stronger assumption than that required by CP: CP relies only on exchangeability between calibration and test data, whereas DRO requires an ambiguity set that contains the true distribution, whose construction from finite data generally requires additional distributional assumptions. DRO then optimizes against the worst-case distribution in this set, providing robustness to finite-sample estimation error and distribution shift. Typical ambiguity sets are defined through moments \cite{Parys2015monent, delage2010distributionally}, statistical divergences such as Kullback--Leibler distance \cite{Jiang2016DatadrivenCC}, or probability metrics such as the Wasserstein distance \cite{Esfahani2018DatadrivenDR, Xie2021OnDR}. Among these, Wasserstein DRO has become particularly influential because it often combines computational tractability with finite-sample guarantees. Notably, Mohajerin Esfahani and Kuhn \cite{Esfahani2018DatadrivenDR} developed a data-driven Wasserstein DRO formulation with rigorous out-of-sample guarantees when the true distribution is unknown. Two questions are central to this formulation: how to evaluate the worst case over the ambiguity set, and how to choose the size of the set. Strong-duality results \cite{blanchet2019quantifying,gao2023distributionally} address the first by replacing the infinite-dimensional optimization over distributions with a tractable dual problem. Concentration results for the empirical measure \cite{fournier2015rate} address the second by specifying the radius of a Wasserstein ambiguity ball that contains the true distribution with prescribed confidence. Comprehensive treatments are given in \cite{rahimian2022frameworks, kuhn2025distributionally}. These properties have made DRO an increasingly important tool in machine learning \cite{blanchet2019robust, chen2020distributionally, gao2023finite}, control \cite{long2023dro_lf,Dimitris_2021_TAC, hakobyan2024wasserstein}, and robotics \cite{coulson2021distributionally, ren2022distributionally_ral, hakobyan2022distributionally}, where uncertainty is often data-driven.

In machine learning, DRO has been used both for robust training and for understanding existing learning objectives. Blanchet et al.~\cite{blanchet2019robust} showed that several standard estimators, including LASSO and logistic regression, can be recast as Wasserstein DRO problems, revealing a connection between robustness and regularization. Building on this perspective, \cite{chen2020distributionally} developed a distributionally robust learning framework under Wasserstein ambiguity sets, encompassing regression, classification, and sequential decision-making problems. Chen and Paschalidis~\cite{chen2018robust} showed that Wasserstein DRO yields tractable formulations with guarantees against adversarial outliers. Gao~\cite{gao2023finite} provided finite-sample guarantees for broad classes of Wasserstein DRO learning problems without suffering from the curse of dimensionality, showing that the ambiguity radius balances empirical loss against loss variation. More recently, \cite{wu2023understanding} interpreted contrastive representation learning through the lens of DRO, showing that its training objective implicitly performs worst-case optimization over negative samples and using this to explain robustness to sampling bias.

% \marginJC{We have 1 paragraph on DRO in machine learning and 1 on DRO in robotics. I think we should also have 1 paragraph on DRO in control, e.g.,~\cite{AC-JC:20-tac,DB-JC-SM:21-tac,IY:21}, particularly of power systems~\cite{YG-KB-EDA-ZH-THS:19,BL-RJ-JLM:19,BKP-ARH-SB-DSC-AC:21}. If you chase them down, you'll find more  refs that might be relevant.}

In control, DRO has been used to design controllers that are robust to distributional uncertainty in disturbances, dynamics, and constraints. One line of work builds data-driven ambiguity sets with finite-sample guarantees for control, in cooperative \cite{AC-JC:20-tac} and time-varying \cite{DB-JC-SM:21-tac} settings, and applies them to stochastic optimal control \cite{IY:21, Parys2015monent} and to safe and stable control synthesis via robust barrier and Lyapunov certificates \cite{PM-KL-NA-JC:23-csl, long2023safe_dro_acc, long2024distributionally}. DRO has been used in model predictive control to provide closed-loop guarantees \cite{mcallister2025distributionally,coulson2021distributionally}, recursive feasibility from data \cite{mark2023recursively} with total-variation and tube-based formulations \cite{dixit2022distributionally, aolaritei2023wasserstein}. DRO has also been impactful in power and energy systems, where distributionally robust chance constraints bound constraint-violation risk under the unknown distribution of renewable generation and load, with applications to optimal power flow \cite{YG-KB-EDA-ZH-THS:19, BL-RJ-JLM:19, duan2018distributionally, xie2018distributionally} and economic dispatch \cite{BKP-ARH-SB-DSC-AC:21}.

In autonomous systems and robotics, DRO has been used to make planning and control robust to uncertainty in state estimation, perception, and environment models and changes. Summers~\cite{summers2018_dr_rrt} proposed a distributionally robust sampling-based planner that accounts for localization, dynamics, and obstacle uncertainties through worst-case risk allocation, while \cite{lathrop2021distributionally} developed a rapidly exploring random tree algorithm that leverages Wasserstein ambiguity sets to provide finite-sample probabilistic safety guarantees in uncertain obstacle environments. Boskos et al. \cite{DB-JC-SM:23-cdc} proposed a distributionally robust coordination algorithm to achieve optimal deployment of a multi-agent system, responding to events of interest with unknown probability distribution.
Xu et al. \cite{xu2024_drcc} incorporated Wasserstein distributionally robust chance constraints into safe-corridor trajectory optimization and derived a convex quadratic reformulation. DRO has also been increasingly applied to safe navigation under uncertainty, including sensing and localization errors \cite{long2024sensorbased_dro}, uncertain obstacle and pedestrian motion \cite{Hakobyan2022_collision_cvar,ryu_2024icra_dro_safe}, and uncertainty from learned perception or environment models \cite{ham_dro_edl_mpc_2025,long2025neural_cdf_dro}.

\subsubsection*{Unified perspective}

Recent work has started to connect CP and DRO. Cauchois et al.~\cite{cauchois2024robust} developed a robust validation approach that uses CP-style prediction sets to guarantee coverage over an $f$-divergence ambiguity set around the source distribution, thereby linking distribution shift robustness to worst-case coverage control. Aolaritei et al.~\cite{aolaritei2025conformal} further developed a distributionally robust form of CP under Lévy-Prokhorov ambiguity sets, showing how local and global perturbations of the test distribution can be propagated through conformity scores to characterize worst-case quantiles and coverage. Guo~\cite{guo2026learning} utilized CP in a DRO formulation by centering the DRO ambiguity set at a fitted predictive distribution, rather than the empirical one, and calibrating the DRO radius using split CP across different problem instances. 

%split CP across problem instances calibrates its radius so that the set contains the unknown distribution with prescribed probability. 

CP and DRO are not the only routes to distribution-free finite-sample guarantees. Classical distribution-free tolerance limits provide coverage from order statistics, with the  required sample size following from a binomial tail
\cite{wilks1941determination}. The scenario approach \cite{calafiore2006scenario,
campi2008exact} carries this idea into optimization: the solution of a convex program built from $K$ sampled constraints violates the true constraint with a known probability, and a sampling-and-discarding refinement, in which some of the sampled constraints are removed, yields a violation probability with a Beta
distribution \cite{campi2011sampling}. More recently, wait-and-judge guarantees have been extended beyond convex programs to the nonconvex setting~\cite{garatti2024nonconvex}. O'Sullivan et al.~\cite{NO-LR-KM:25} show that ranking nonconformity scores is a one-dimensional scenario program with discarded constraints, and extend the argument to calibration-conditional CP. Calafiore~\cite{calafiore2026bridging} further develops the connection between conformal calibration and scenario optimization, using it to allocate risk across multiple calibrated constraints. Broader comparisons of CP with scenario optimization and PAC-Bayes theory for verification and control are provided in~\cite{lindemann2025formal}.

Unlike prior works that primarily combine CP with distributionally robust ideas to obtain shift-robust coverage, we aim to develop a unified probabilistic framing of CP and DRO, characterize their connection in finite-sample quantile estimation, and provide a systematic comparison of their statistical behavior and practical trade-offs.

%% file: tex/ProblemFormulation.tex
\section{Problem Formulation and Preliminaries}
\label{sec: UQ}

Let $(\Omega, \mathcal{F}, P)$ be a probability space, where $\Omega$ is the sample space, 
$\mathcal{F}$ is a $\sigma$-algebra, and $P$ is a probability measure. For $\mathcal{X} \subseteq \mathbb{R}$ with Borel $\sigma$-algebra $\mathcal{B}(\mathcal{X})$, 
denote by $\mathcal{P}(\mathcal{X})$ the set of all Borel probability measures on $\mathcal{X}$. Let $\mathbb{P} \in \mathcal{P}(\mathcal{X})$ denote the distribution of a random variable $R: \Omega \to \mathcal{X}$.
We use $\mathbb{P}_{\text{test}} \in \mathcal{P}(\mathcal{X})$ to denote a test distribution 
that may differ from $\mathbb{P}$. Let $R^{(0)}:\Omega \to \calX
$, referred to as the test score, be drawn from $\bbP_{\text{test}}$. 

For any distribution $\mathbb{Q}\in\mathcal{P}(\mathcal{X})$ and level
$p\in(0,1)$, define its $p$-quantile as
\begin{equation} \label{eq:quantile_definition} q_{p}(\mathbb{Q}) := \inf\{ \alpha\in\mathbb{R}: \mathbb{Q}(R\leq\alpha)\geq p \}, \;\; R\sim\mathbb{Q}.
\end{equation} 
For $n$ values $a_1,\ldots,a_n\in\mathbb{R}\cup\{\infty\}$ and a level
$p\in(0,1)$, we write $\Quant_p(a_1,\ldots,a_n)$ for the $\ceil*{np}$-th smallest
of them, the empirical $p$-quantile.
If $\bbP_{\text{test}}$ were known, the solution would be
$q_{1-\delta}(\bbP_{\text{test}})$ in \eqref{eq:quantile_definition}, available in
closed form for standard families such as the Gaussian or the uniform. In
practice $\bbP_{\text{test}}$ is unknown and $\alpha$ must be estimated from a
calibration dataset $R^{(1)}, \hdots, R^{(K)}$ drawn i.i.d.\ from $\bbP$,
denoted $R^{1:K}:= R^{(1)},\ldots,R^{(K)}$. Since $\bbP$ may differ from
$\bbP_{\text{test}}$, estimation requires a known bound on their discrepancy.

\begin{problem}\longthmtitle{Test-distribution quantile estimation}
\label{prob: UQ}
Let $\delta \in (0,1)$ be a risk level, let $R^{1:K}$ be calibration scores
drawn i.i.d.\ from $\bbP$, and let $\dist$ be a discrepancy on
$\mathcal{P}(\mathcal{X})$ with a known budget $\gamma \ge 0$ such that
$\dist(\bbP,\bbP_{\text{test}}) \le \gamma$. Given $\delta$, $R^{1:K}$ and
$(\dist,\gamma)$, compute the smallest $\alpha \in \mathbb{R}$ satisfying:
\begin{equation}\label{eq:original_UQ}
    \bbP_{\text{test}}(R^{(0)} \leq \alpha) \geq 1 - \delta .
\end{equation}
\end{problem}

We first consider the case $\bbP = \bbP_{\text{test}}$, and aim to construct an estimator $\bar{\alpha} = \bar{\alpha}(R^{1:K})$ such that $\mathbb{P}(R^{(0)} \leq \bar{\alpha}) \geq 1 - \delta$ holds with high probability over the draw of the calibration samples. We define the problem as follows.

%
% \marginJC{I wonder if we could label Problem 2 as 1.a and Problem 3 as 1.b? That way is clear that Problems 2 and 3 are just different flavors of Problem 1. With the current numbering, in what follows, we deal with Problem 2 and Problem 3, but never refer back to Problem 1. With the suggested numbering, i think it is clearer that we're just dealing with tractable formulations of Problem 1.}
%

\begin{subproblem}\longthmtitle{Calibration-conditional quantile estimation}
\label{prob:UQ_no_test_shift}
Consider Problem~\ref{prob: UQ} with $\gamma = 0$, i.e., $\bbP_{\text{test}} = \bbP$.
Given a risk level $\delta \in (0,1)$ and $K+1$ i.i.d. samples
$R^{(0)},R^{(1)},\ldots,R^{(K)}$ from $\mathbb{P}$, construct a
data-dependent threshold
$\bar{\alpha}=\bar{\alpha}(R^{1:K})$
such that
\begin{align}
\label{eq:calib_cond_gua}
    \mathbb{P}\left(
    R^{(0)} \leq \bar{\alpha}
    \; \middle|\, R^{1:K}
    \right)
    \geq 1-\delta .
\end{align}
\end{subproblem}

We refer to \eqref{eq:calib_cond_gua} as a calibration-conditional coverage guarantee. The terminology emphasizes that, after the calibration samples $R^{1:K}$ are realized, the data-dependent threshold $\bar{\alpha}$ should cover an independent future test score $R^{(0)}$ with probability at least $1-\delta$.

%The formulation above assumes that the calibration and test scores are drawn from the same distribution. In practice, this assumption may fail when calibration data are historical or collected under different environments, sensing conditions, populations, or tasks than the future test data. 

The assumption that the calibration and test scores are drawn from the same distribution may fail in practice, e.g., when the calibration data are historical or collected under different environments, sensing conditions, populations, or tasks than the future test data. For example, a model calibrated on past observations may be deployed under a new temporal regime, or a language model calibrated on one benchmark or prompt distribution may be evaluated on another. 
%Consequently, the test-score distribution may differ from the calibration-score distribution, i.e., $\mathbb{P}_{\mathrm{test}} \neq \mathbb{P}$. 
To account for such calibration-test distribution shift, we consider Problem~\ref{prob:UQ_test_shift}.

\begin{subproblem}\longthmtitle{Calibration-test shift-robust quantile estimation}
\label{prob:UQ_test_shift}
Consider Problem~\ref{prob: UQ} with $\gamma > 0$.
Given a risk level $\delta \in (0,1)$, $K$ i.i.d. calibration samples
$R^{1:K}$ drawn from $\mathbb{P}$, and an independent test sample $R^{(0)}$ drawn from $\mathbb{P}_{\mathrm{test}}$, construct
an estimator $\bar{\alpha}=\bar{\alpha}(R^{1:K})$
such that
\begin{align}
\label{eq:calib_cond_gua_shift}
    \mathbb{P}_{\mathrm{test}}\left(
    R^{(0)} \leq \bar{\alpha}
    \; \middle|\,
    R^{1:K}
    \right)
    \geq 1-\delta,
\end{align}
where $\mathbb{P} \neq \mathbb{P}_{\mathrm{test}}$.
\end{subproblem}

%
% \marginJC{Can we add informal names to Problems 1-3?}
%

In this paper, we present two methods for constructing the estimators in \eqref{eq:calib_cond_gua} and \eqref{eq:calib_cond_gua_shift}: conformal prediction (CP) and distributionally robust optimization (DRO).

% \NA{We can simplify the section titles below, e.g.,
% \begin{itemize}
%     \item Sec. 4. CP and DRO Without Distribution Shift
%     \item Sec. 5. CP and DRO With Distribution Shift
% \end{itemize}}

\section{CP and DRO Without Distribution Shift}

In this section, we address Problem~\ref{prob:UQ_no_test_shift} with CP and DRO when no test-time distribution shift is present. We first consider the CP approach. 

\subsection{Calibration-Conditional Conformal Prediction} 
\label{sec: uq_cp} 

Unfortunately, exact calibration-conditional guarantees of the form \eqref{eq:calib_cond_gua} cannot be achieved in a finite-sample, distribution-free setting. Intuitively, after conditioning on a finite calibration set, the observed samples do not fully determine the unseen tail of the distribution; therefore, any finite threshold may fail to cover enough probability mass for some distribution consistent with the calibration data. Instead, \cite{vovk2012conditional} presents a conformal prediction variant that provides calibration-conditional coverage guarantees of the form:
\begin{align}\label{eq:cal_cond_goal}
    \mathbb{P}^K\big(\mathbb{P}(R^{(0)}\le \bar{\alpha}(R^{1:K}))\ge 1-\delta\big)\ge 1-\beta,
\end{align}
where $\beta\in (0,1)$ is a user-defined  failure probability over the calibration data. For notational convenience, we write $\mathbb{P}^K = \mathbb{P}^{\otimes K}$ 
for the product measure of $K$ i.i.d. draws from $\mathbb{P}$, 
so that $\mathbb{P}^K$ denotes probability with respect to the samples $\{R^{1:K}\}$, and $\mathbb{E}^K$ the corresponding expectation over $R^{1:K}\sim\mathbb{P}^K$. We intuitively interpret the statement in equation \eqref{eq:cal_cond_goal} as ``with probability no less than $1-\beta$ over the draw of calibration datapoints $R^{1:K}$, it holds that $R^{(0)}\le \bar{\alpha}(R^{1:K})$ with probability no less than $1-\delta$ over the draw of a test datapoint $R^{(0)}$.'' Consequently, the outer probability measure $\mathbb{P}^K(\cdot)$ is defined over the randomness in $R^{1:K}$, while the inner probability measure $\mathbb{P}(\cdot)$ is defined over the randomness in $R^{(0)}$. If $\beta\in (0,1)$ is chosen to be very small, then one can approximately obtain calibration-conditional guarantees. 

Classical split conformal prediction \cite{shafer2008tutorial, angelopoulos2021gentle} is typically formulated through the marginal coverage guarantee:
\begin{equation}
\label{eq:marginal_coverage_goal}
\mathbb P^{K+1}\left(
R^{(0)}\le \bar\alpha(R^{1:K})
\right)\ge 1-\delta,
\end{equation}
which averages jointly over the randomness in the calibration sample
$R^{1:K}$ and the independent test score $R^{(0)}$. Consequently, marginal
coverage does not control the coverage obtained for a particular
realization of the calibration sample: calibration sets with coverage
below $1-\delta$ may be compensated for by calibration sets with higher
coverage.

The calibration-conditional guarantee in \eqref{eq:cal_cond_goal}
provides additional control by requiring that the target coverage holds for at least a $1-\beta$ fraction of calibration samples. The following result relates this two-level guarantee to marginal coverage.

\begin{lemma}\longthmtitle{Calibration-conditional coverage implies marginal coverage}
\label{lem:marginal_coverage}
Let $R^{(0)},R^{(1)},\ldots,R^{(K)}$ be i.i.d.\ draws from a distribution $\mathbb P$, and let $\bar\alpha:\mathbb R^K\to\mathbb R$ be any measurable function of the calibration samples $R^{1:K}$. 
If for some $\delta,\beta\in(0,1)$:
\[
\mathbb P^K \bigl(\ \mathbb P\bigl(R^{(0)}\le \bar\alpha(R^{1:K}) \bigr) \ge 1-\delta\ \bigr) \ge 1-\beta,
\]
then the marginal coverage satisfies: % (average-case) 
\[
\mathbb{P}^{K + 1} \bigl(R^{(0)}\le \bar\alpha(R^{1:K})\bigr) \ge (1-\delta)(1-\beta),
\]
where $\mathbb{P}^{K + 1}$ is the $K + 1$ dimensional product measure. 
\resultend
\end{lemma}

\begin{proof}
Define the calibration-conditional coverage
\[
c(R^{1:K}) := \mathbb P\bigl(R^{(0)}\le \bar\alpha(R^{1:K})\bigr)\in[0,1]
\]
and the event $A:=\{\,c(R^{1:K})\ge 1-\delta\,\}$. By the tower property,
\begin{align}
&\mathbb{P}^{K + 1} \bigl(R^{(0)}\le \bar\alpha(R^{1:K})\bigr)
= \mathbb E^K\bigl[c(R^{1:K})\bigr] \notag \\
&\quad = \mathbb E^K\bigl[c(R^{1:K})\,\mathds{1}_A\bigr]
+\mathbb E^K\bigl[c(R^{1:K})\,\mathds{1}_{A^c}\bigr] \notag \\
&\quad \ge (1-\delta)\,\mathbb P^K(A) + 0\cdot \mathbb P^K(A^c).
\end{align}
Using $\mathbb P^K(A)\ge 1-\beta$ gives the claim.
\end{proof}

Lemma~\ref{lem:marginal_coverage} shows that the two-level guarantee also
provides marginal coverage, although at the lower bound $(1-\delta)(1-\beta)$. The converse need not hold: a method may achieve
the marginal guarantee in \eqref{eq:marginal_coverage_goal} while failing to attain coverage $1-\delta$ for a non-negligible fraction of
calibration samples. This distinction is central to our comparison and
is illustrated in the numerical studies. We next summarize how calibration-conditional conformal prediction constructs an estimator
$\bar\alpha$ satisfying \eqref{eq:cal_cond_goal}.

The estimators below share the same reduction. Each selects an order statistic from
$\{R^{1:K},\infty\}$. Under continuity of the score distribution, the selected threshold achieves
coverage at least $1-\delta$ if and only if it is at least the true $(1-\delta)$-quantile $\alpha$.
For a fixed order statistic, this depends only on how many of the $K$ calibration scores fall
below $\alpha$. That count is $\mathrm{Binomial}(K,1-\delta)$, so \eqref{eq:cal_cond_goal} reduces
to controlling a binomial tail. The three constructions below differ only in how they bound it.

%
% \marginJC{Are we doing 1st letter=capitals in titles? Seems like most of the times, but not consistently always}
%
\begin{lemma}[\textbf{Calibration-conditional coverage} \cite{vovk2012conditional}]\label{lem:2}
    Let $R^{(0)},\hdots,R^{(K)}$ be $K+1$ independent and identically distributed random variables. Then, the probabilistic prediction region in equation \eqref{eq:cal_cond_goal} is valid for the choice:
\begin{equation}\label{eq:cp_cond}
    \bar{\alpha}_{\mathrm{CP}}:=\Quant_{1 - \delta + \sqrt{\frac{\ln(1/\beta)}{2K}}}( R^{1:K}, \infty ).
\end{equation}
\resultend
\end{lemma}

% \NA{Minor but is there a way to avoid having \resultend alone on a new line?}

Lemma~\ref{lem:2} follows from Hoeffding's inequality~\cite{bentkus2004hoeffding}, which bounds the tail without using the variance $\delta(1-\delta)$, so its correction does not shrink as $\delta$ does. To mitigate this,~\cite{vovk2012conditional, angelopoulos2024theoretical} proposed two alternative formulations for estimating $\bar{\alpha}$: Lemma~\ref{lem:3} uses the exact binomial tail, and Lemma~\ref{lem:4} a Bernstein bound~\cite{boucheron2013concentration} that retains the variance. 

\begin{lemma}[\textbf{Calibration-conditional coverage, variant 1} \cite{vovk2012conditional, angelopoulos2024theoretical}]
\label{lem:3}
Under the setting of Lemma~\ref{lem:2}, the probabilistic prediction region in equation \eqref{eq:cal_cond_goal} is valid for the choice:
\begin{equation}\label{eq: cali_cp_lemma2} \bar{\alpha}_{\mathrm{CP}2}:=\Quant_{1 - \delta'}( R^{1:K}, \infty ) \end{equation}
where $\delta' \in [0,1]$ satisfies $\beta \ge I_{1-\delta}(K-\lfloor k \rfloor, 1+\lfloor k \rfloor)$,
with $k = \delta'(K+1)-1$, and
$I_x(a,b)$ denoting the regularized incomplete beta function,
i.e., the cumulative distribution function of a $\mathrm{Beta}(a,b)$ distribution
evaluated at $x$.
\resultend
\end{lemma}

\begin{lemma}[\textbf{Calibration-conditional coverage, variant 2} \cite{vovk2012conditional}]
\label{lem:4}
Under the setting of Lemma~\ref{lem:2}, the probabilistic prediction region in equation \eqref{eq:cal_cond_goal} is valid for the choice:
\begin{equation}\label{eq: cali_cp_lemma3} \bar{\alpha}_{\mathrm{CP}3}:=\Quant_{1 -\delta + \sqrt{\frac{2 \delta \ln(1/\beta)}{K}} + \frac{2\ln(1/\beta)}{K}}( R^{1:K}, \infty ). 
\end{equation}
\resultend
\end{lemma}

The connection between Lemma~\ref{lem:3} and scenario optimization is
established in~\cite{NO-LR-KM:25}: the $j$-th order statistic of $K$ scores
solves a one-dimensional scenario program with $K-j$ of the sampled constraints
discarded, and the beta-function condition of Lemma~\ref{lem:3} is the
sampling-and-discarding guarantee of~\cite{campi2011sampling} specialized to
that program. The correspondence is exact only for Lemma~\ref{lem:3}, since
Lemmas~\ref{lem:2} and~\ref{lem:4} bound the same condition through Hoeffding's
and Bernstein's inequalities and are therefore conservative relative to it.

\begin{remark}[\textbf{Finite-sample non-vacuity}]
\label{rem:cp_finite_sample}
For an estimator of the form
$\Quant_{p_K}(R^{1:K},\infty)$, the appended value
$\infty$ is inactive, and hence the estimator is finite, if and only if
\begin{equation}
p_K\le \frac{K}{K+1}.
\end{equation}
If $p_K>\frac{K}{K+1}$, the quantile selects the value
$\infty$, yielding a vacuous prediction region. For the estimator in Lemma~\ref{lem:2}, the finite-threshold condition is
\begin{equation}
\label{eq:cp_finite_condition}
\sqrt{\frac{\ln(1/\beta)}{2K}}+\frac{1}{K+1}\le\delta.
\end{equation}
Equivalently, the minimum calibration size is
\begin{equation}
K_{\min}^{\mathrm{CP}}(\delta,\beta)
:=
\min\left\{
K\in\mathbb N:
\sqrt{\frac{\ln(1/\beta)}{2K}}+\frac{1}{K+1}\le\delta
\right\}.
\end{equation}
For Lemma~\ref{lem:3}, the threshold is the $j$-th order statistic with
$j=K-\lfloor k\rfloor$, and $I_{1-\delta}(j,K+1-j)$ is decreasing in $j$. A finite threshold is therefore available when the largest admissible choice $j=K$ satisfies the beta-function condition. Since
$\mathrm{Beta}(K,1)$ has cumulative distribution function $x^{K}$, this leads to:
\begin{equation} \label{eq:cp3_finite_condition}
(1-\delta)^{K}\le\beta,
\;\text{equivalently}\;
K^{\mathrm{CP}2}_{\min}(\delta,\beta)
=\left\lceil\frac{\ln(1/\beta)}{\ln\big(1/(1-\delta)\big)}\right\rceil.
\end{equation}
For Lemma~\ref{lem:4}, the corresponding condition is
\begin{equation}
\sqrt{\frac{2\delta\ln(1/\beta)}{K}}
+\frac{2\ln(1/\beta)}{K}
+\frac{1}{K+1}
\le\delta.
\end{equation}
\resultend
\end{remark}

The variant in Lemma~\ref{lem:4} can be sharper than Lemma~\ref{lem:2}
when $\delta$ is small, and the exact beta-function construction of
Lemma~\ref{lem:3} is sharper than both. Table~\ref{tab:kmin} makes the
comparison precise. The Hoeffding-based
condition~\eqref{eq:cp_finite_condition} forces
$K_{\min}=\Theta\big(\ln(1/\beta)/\delta^{2}\big)$, whereas
Lemmas~\ref{lem:3} and~\ref{lem:4} both scale as
$\Theta\big(\ln(1/\beta)/\delta\big)$, with a much smaller constant for
Lemma~\ref{lem:3} by~\eqref{eq:cp3_finite_condition}. These conditions suggest a small-sample limitation of the constructions
above: when the required confidence correction is large relative to the calibration size, the only distribution-free threshold they supply is the vacuous value $\infty$. Table~\ref{tab:kmin} summarizes the three lemmas: Lemma~\ref{lem:3} is non-vacuous at a much smaller $K$. All three are evaluated in Sections~\ref{sec:examples} and~\ref{sec:exp_setup_cp_cls}.

\begin{table}[htb]
    \centering
    \small
    \caption{Minimum calibration size $K_{\min}$ for a finite
    (non-vacuous) threshold, at $\beta=0.05$; the smallest of the three
    is shown in bold. The Hoeffding bound of Lemma~\ref{lem:2} degrades
    as $\delta^{-2}$, while Lemmas~\ref{lem:3} and~\ref{lem:4} degrade
    as $\delta^{-1}$.}
    \label{tab:kmin}
    \begin{tabular}{@{}lrrr@{}}
    \toprule
    $\delta$ & Lemma~\ref{lem:2} & Lemma~\ref{lem:3} & Lemma~\ref{lem:4} \\
    \midrule
    $0.2$   & $47$          & $\mathbf{14}$    & $86$ \\
    $0.1$   & $170$         & $\mathbf{29}$    & $172$ \\
    $0.05$  & $639$         & $\mathbf{59}$    & $343$ \\
    $0.01$  & $15178$     & $\mathbf{299}$   & $1712$ \\
    $0.001$ & $1499866$ & $\mathbf{2995}$ & $17119$ \\
    \bottomrule
    \end{tabular}
\end{table}

We next turn to distributionally robust optimization (DRO), which supplies a threshold that is finite for every $K$. Rather than relying on a finite-sample calibration argument, DRO constructs a worst-case guarantee over an ambiguity set of probability distributions centered at the empirical distribution $\widehat{\bbP}_K= \frac{1}{K} \sum_{i=1}^K \delta_{R^{(i)}}$, where $\delta_x$ denotes the Dirac measure at $x$; by construction $q_{1-\delta}(\widehat{\bbP}_K)=\Quant_{1-\delta}(R^{1:K})$. It attains the same calibration-conditional guarantee, at the cost of the distributional assumptions needed to certify its radius.
%
% \marginJC{I think we need a different notation for the empirical distribution: as it stands, $\bbP_K$ can easily be confused with $\bbP^K$, and these are very different things.}
%

\subsection{Distributionally Robust Optimization} 
To formalize the DRO construction, we introduce a Wasserstein ambiguity set that defines a set of possible distributions around the empirical distribution. For $\mathbb{P}, \mathbb{Q} \in \mathcal{P}(\mathcal{X})$, we denote by $\Gamma(\mathbb{P},\mathbb{Q})$ the set of couplings, i.e., all joint distributions on $\mathcal{X} \times \mathcal{X}$ with marginals $\mathbb{P}$ and $\mathbb{Q}$. We consider the $\infty$-Wasserstein distance:
\begin{equation}
\label{eq: wasserstein_infty_distance}
  W_\infty(\mathbb{P}, \mathbb{Q}) := 
  \inf_{\pi \in \Gamma(\mathbb{P},\mathbb{Q})}
    \sup_{(x_1, x_2) \in \operatorname{supp}(\pi)} \|x_1 - x_2\|.
\end{equation}
Here $\operatorname{supp}(\pi)$ denotes the support of the coupling  $\pi$, i.e., the smallest closed set $S \subseteq \mathcal{X}\times\mathcal{X}$ such that $\pi(S) = 1$. Given $\widehat{\bbP}_K$, we define an ambiguity set of distributions $\bbQ$ that are within $W_\infty$ distance of $r>0$ from $\widehat{\bbP}_K$:
\begin{equation}
\label{eq: wasserstein_ball_def}
\mathcal{M}^{r}(\widehat{\bbP}_K) := \left\{ \mathbb{Q} \in \mathcal{P}(\mathbb{X}) \,\middle|\, W_{\infty}(\mathbb{Q}, \widehat{\bbP}_K) \le r \right\}.
\end{equation}
We use $\calM^r(\widehat{\bbP}_K)$ to construct a DRO-based quantile bound analogous to the conformal prediction result.

A key step in the DRO construction is to choose the ambiguity radius $r$ so that the Wasserstein ball around the empirical distribution contains the true distribution with high probability. Following \cite{fournier2015rate,liu2019rate}, Lemma~\ref{lem:choice-r} provides one such choice, $r=r_K(\beta)$, which ensures this containment with probability at least $(1-\beta)$.

\begin{lemma}\longthmtitle{Choice of Wasserstein-$\infty$ ball radius in one dimension}
\label{lem:choice-r}
Assume $\mathbb{P}$ is supported on $[a,b] \subset \mathbb{R}$ and has
density $f$ satisfying $\inf_{x\in[a,b]}f(x) \ge m>0$. Then, for
any $\beta \in (0,1)$, the choice
\begin{equation}
\label{eq:wasserstein_inf_radius}
    r_K(\beta)=\frac{1}{m}\sqrt{\frac{\ln(2/\beta)}{2K}}
\end{equation}
ensures
\begin{equation}
\label{eq: wasserstein_radius_guarantee}
    \mathbb{P}^K\!\left(
    W_\infty(\widehat{\bbP}_K,\mathbb{P}) \le r_K(\beta)
    \right)
    \ge 1-\beta .
\end{equation}
\resultend
\end{lemma}

The radius in \eqref{eq:wasserstein_inf_radius} is one of several finite-sample choices of the Wasserstein radius proposed in the literature; more general high-confidence ambiguity sets in arbitrary dimension \cite{Dimitris_2021_TAC} and measure-concentration bounds for light-tailed distributions \cite{Esfahani2018DatadrivenDR, fournier2015rate} provide alternatives under different metrics and tail assumptions. These choices, however, share a practical limitation: the radius depends on problem-specific constants (e.g., the dimension, the support diameter, a density lower bound, or tail and moment constants) that are conservative and rarely computable in practice. Consequently, the Wasserstein radius is, in practice, selected by validation or treated as a tunable robustness parameter rather than computed from a closed-form bound.

With \eqref{eq:wasserstein_inf_radius}, the ambiguity set $\mathcal{M}^{r_K(\beta)}$ contains $\mathbb{P}$ with probability at least $1-\beta$, which we use to derive the following distributionally robust conditional coverage guarantee under the $W_{\infty}$ distance.

\begin{lemma}[\textbf{Distributionally robust conditional coverage under $W_\infty$}] 
\label{lem:dro_quantile_refined}
Let $R^{(0)}, R^{(1)}, \ldots, R^{(K)}$ be $K + 1$ independent and identically distributed random variables. Then, the probabilistic prediction region in equation \eqref{eq:cal_cond_goal} is valid for the choice:
\begin{equation}
\label{eq: alpha_dro_bound}
\bar{\alpha}_{\mathrm{DRO}} := \Quant_{1 - \delta}(R^{1:K}) + r_K(\beta), 
\end{equation}
where the Wasserstein ball $\mathcal{M}^{r_K(\beta)}(\widehat{\bbP}_K)$ is defined as in \eqref{eq: wasserstein_ball_def} and the radius \(r_K(\beta)\) is chosen so that \eqref{eq: wasserstein_radius_guarantee} is satisfied. Lemma~\ref{lem:choice-r} provides one choice of such a radius. Moreover, for every realization of the calibration sample,
\begin{equation}
\label{eq:dro_full_ball_coverage}
\inf_{\mathbb{Q}\in \mathcal{M}^{r_K(\beta)}(\widehat{\bbP}_K)}
\mathbb{Q}(R^{(0)} \le \bar{\alpha}_{\mathrm{DRO}})
\ge 1 - \delta.
\end{equation}
\resultend
\end{lemma}

\begin{proof}
We begin by noting that, by the choice of $r_K(\beta)$, the true distribution $\bbP$ lies within the Wasserstein ball centered at $\widehat{\bbP}_K$ with probability at least $1-\beta$:
\begin{equation}
\label{eq:proof_dro_quantile_in_ball}
\bbP^K\left(
\bbP\in\mathcal{M}^{r_K(\beta)}(\widehat{\bbP}_K)
\right)\ge 1-\beta.
\end{equation}
Next, we show that the $p$-quantile functional on probability measures over $\mathbb{R}$ is 1-Lipschitz with respect to the Wasserstein distance $W_{\infty}$.
In one dimension, we have the exact representation
\begin{equation}
\label{eq:proof_dro_Winf_quantiles}
  W_{\infty}(\mathbb{P}_1,\mathbb{P}_2)
  \;=\;
  \sup_{u\in(0,1)}
    \bigl|q_u(\mathbb{P}_1)-q_u(\mathbb{P}_2)\bigr|,
\end{equation}
so for every fixed $u\in(0,1)$,
\begin{equation}
\label{eq:proof_dro_quantile_lipschitz_infty}
  \bigl|q_u(\mathbb{P}_1)-q_u(\mathbb{P}_2)\bigr|
  \;\le\;
  W_{\infty}(\mathbb{P}_1,\mathbb{P}_2).
\end{equation}
Setting $u = 1 - \delta$ and applying this to $\mathbb{P}$ and the empirical distribution $\widehat{\bbP}_K$ yields:
\begin{equation}
\label{eq:proof_dro_quantile_shift_bound}
\left| q_{1-\delta}(\mathbb{P}) - q_{1-\delta}(\widehat{\bbP}_K) \right| \le W_{\infty}(\mathbb{P}, \widehat{\bbP}_K).
\end{equation}
Thus, by \eqref{eq:proof_dro_quantile_in_ball} and the definition
of $\bar{\alpha}_{\mathrm{DRO}}$ in \eqref{eq: alpha_dro_bound}, with
probability at least $1-\beta$ over the calibration samples, we have
\begin{equation}
\label{eq:proof_dro_upper_bound}
q_{1-\delta}(\mathbb{P})
\le
q_{1-\delta}(\widehat{\bbP}_K)+r_K(\beta)
=
\bar{\alpha}_{\mathrm{DRO}} .
\end{equation}
By the definition of the quantile, this implies:
\begin{equation}
\label{eq:proof_dro_final_result}
\mathbb{P}^K\left(
\mathbb{P}(R^{(0)} \le \bar{\alpha}_{\mathrm{DRO}})
\ge 1-\delta
\right)
\ge 1-\beta.
\end{equation}

Moreover, the same quantile Lipschitz argument applies to any
$\mathbb{Q}\in\mathcal{M}^{r_K(\beta)}(\widehat{\bbP}_K)$, since
$W_\infty(\mathbb{Q},\widehat{\bbP}_K)\le r_K(\beta)$ implies
\[
q_{1-\delta}(\mathbb{Q})
\le
q_{1-\delta}(\widehat{\bbP}_K)+r_K(\beta)
=
\bar{\alpha}_{\mathrm{DRO}}.
\]
Therefore,
$\mathbb{Q}(R^{(0)}\le \bar{\alpha}_{\mathrm{DRO}})\ge 1-\delta$
for every
$\bbQ\in\mathcal{M}^{r_K(\beta)}(\widehat{\bbP}_K)$, which gives
\eqref{eq:dro_full_ball_coverage}.
\end{proof}

\begin{table*}[t]
\centering
\small
\caption{CP and DRO estimators without distribution shift
($\bbP_{\mathrm{test}}=\bbP$).}
\label{tab:comparison_without_shift}
\setlength{\tabcolsep}{4pt}
\renewcommand{\arraystretch}{1.2}
\resizebox{\linewidth}{!}{
\begin{tabular}{@{}p{4.5cm}p{2.5cm}p{5.5cm}p{5.5cm}p{1.8cm}@{}}
\toprule
\textbf{Method} &
\textbf{Shift model} &
\textbf{Coverage guarantee} &
\textbf{Estimator $\displaystyle\bar\alpha(R^{1:K})$} &
\textbf{Result} \\
\midrule
%---------------------------------------------------------------------
\textbf{Cal.-conditional CP} &
$\bbP_{\mathrm{test}}=\bbP$ &
$\displaystyle
\bbP^K\Bigl[
\bbP\{R^{(0)}\le\bar\alpha\}\ge 1-\delta
\Bigr]\ge 1-\beta$ &
$\displaystyle
\Quant_{\,
1-\delta+\sqrt{\frac{\ln(1/\beta)}{2K}}}
\bigl(R^{1:K},\infty\bigr)$ &
Lemma~\ref{lem:2} \\[8pt]
%---------------------------------------------------------------------
\textbf{Cal.-conditional CP (variant 1)} &
$\bbP_{\mathrm{test}}=\bbP$ &
$\bbP^K\Bigl[
\bbP\{R^{(0)}\le\bar\alpha\}\ge 1-\delta
\Bigr]\ge 1-\beta$ &
$\displaystyle
\Quant_{\,1-\delta'}
\bigl(R^{1:K},\infty\bigr)$ &
Lemma~\ref{lem:3} \\[8pt]
%---------------------------------------------------------------------
\textbf{Cal.-conditional CP (variant 2)} &
$\bbP_{\mathrm{test}}=\bbP$ &
$\bbP^K\Bigl[
\bbP\{R^{(0)}\le\bar\alpha\}\ge 1-\delta
\Bigr]\ge 1-\beta$ &
$\displaystyle
\Quant_{\,
1-\delta+\sqrt{\frac{2\delta\ln(1/\beta)}{K}}
+\frac{2\ln(1/\beta)}{K}}
\bigl(R^{1:K},\infty\bigr)$ &
Lemma~\ref{lem:4} \\[8pt]
\textbf{DRO} &
$\bbP_{\mathrm{test}}=\bbP$ &
$\bbP^K\Bigl[
\bbP\{R^{(0)}\le\bar\alpha\}\ge 1-\delta
\Bigr]\ge 1-\beta$ &
$\displaystyle
\Quant_{1-\delta}(R^{1:K})
+r_K(\beta)$ &
Lemma~\ref{lem:dro_quantile_refined} \\
%---------------------------------------------------------------------
\bottomrule
\end{tabular}}
\end{table*}
%
% \marginJC{Why do we write here (7) and (8) instead of the quantile expressions?}
%

\begin{figure}[t]
    \centering
    \includegraphics[width=\linewidth]{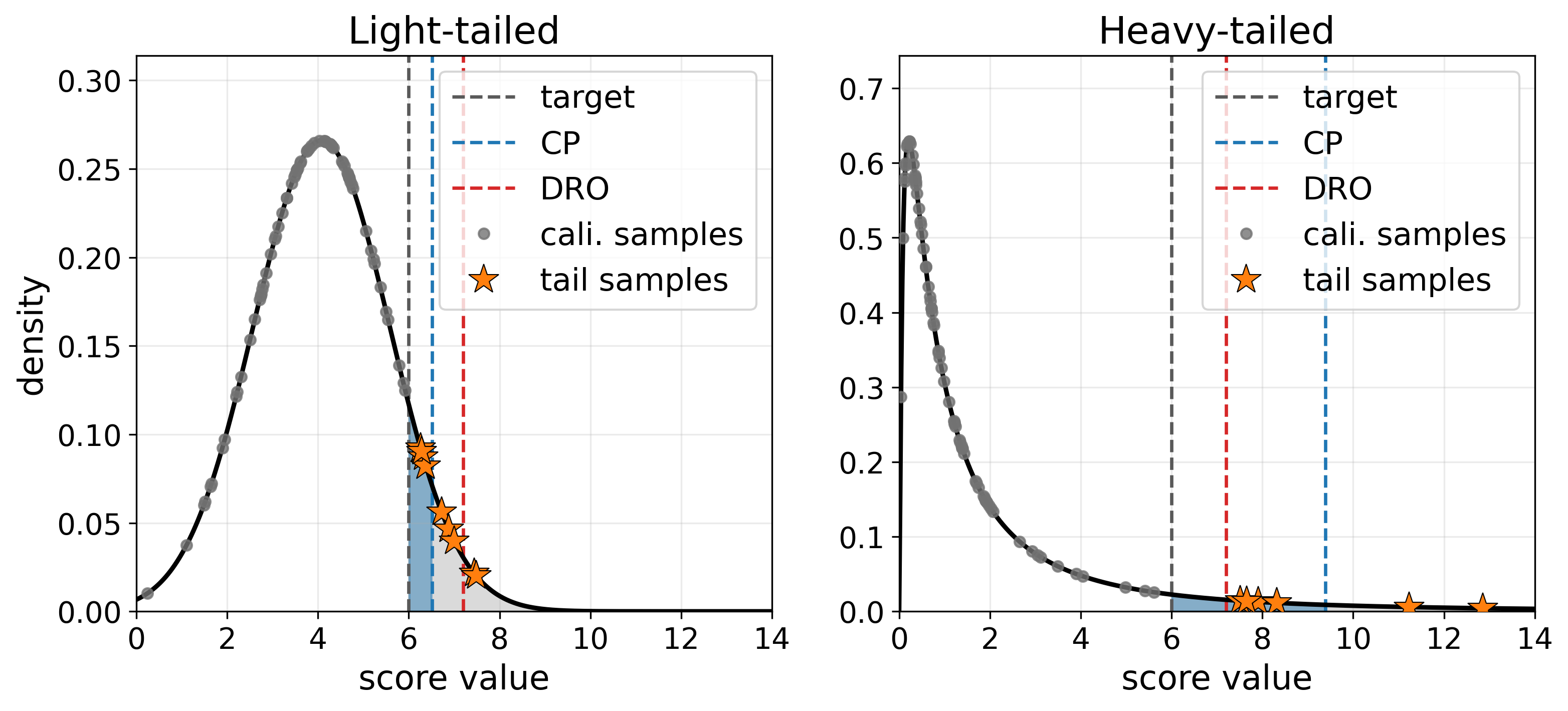}
    \caption{The two corrections act on different axes: CP shifts the quantile \emph{index}, while DRO shifts the quantile \emph{value}. Dots are calibration samples placed on the density at their value; stars are the upper-tail samples ($>q_{1-\delta}$) that CP's level inflation $1-\delta \to 1-\delta+\epsilon$ relies on. For a light-tailed distribution
    (left), these samples cluster near the $(1-\delta)$ quantile, so the CP threshold barely moves; for a heavy-tailed distribution (right), the samples are sparse and far out, pushing the CP threshold deep into the tail, whereas the DRO radius $r$ adds a similar value-space margin in both. The shown radius $r$ is illustrative, not the guaranteed radius $r_K(\beta)$ of Lemma~\ref{lem:choice-r}.}
    \label{fig:cp_dro_intuition}
\end{figure}

Because Lemma~\ref{lem:marginal_coverage} is stated for an arbitrary measurable $\bar\alpha$, it applies to $\bar\alpha_{\mathrm{DRO}}$: the guarantee of Lemma~\ref{lem:dro_quantile_refined} implies the marginal coverage bound $\mathbb{P}^{K+1}(R^{(0)}\le\bar\alpha_{\mathrm{DRO}}(R^{1:K}))\ge(1-\delta)(1-\beta)$.

\begin{remark}\longthmtitle{Robust joint coverage over the ambiguity ball}
  The guarantee in \eqref{eq:dro_full_ball_coverage} holds for every realization of the calibration sample, so it survives taking expectations. Since the ball $\mathcal{M}^{r_K(\beta)}(\widehat{\bbP}_K)$ is centered at the \emph{random} empirical measure $\widehat{\bbP}_K$, the worst-case test distribution must be modeled as a measurable data-dependent selection $R^{1:K}\mapsto\mathbb{Q}_{R^{1:K}}\in\mathcal{M}^{r_K(\beta)}(\widehat{\bbP}_K)$, and
  \[
  \mathbb{E}^K\!\left[
  \inf_{\mathbb{Q}\in\mathcal{M}^{r_K(\beta)}(\widehat{\bbP}_K)}
  \mathbb{Q}\bigl(R^{(0)}\le\bar\alpha_{\mathrm{DRO}}\bigr)\right]
  \;\ge\;1-\delta ,
  \]
  equivalently $\inf_{\mathbb{Q}}\Pr_{R^{1:K},\,R^{(0)}\sim\mathbb{Q}_{R^{1:K}}}\{R^{(0)}\le\bar\alpha_{\mathrm{DRO}}\}\ge 1-\delta$,
  the infimum being over all such selections. No factor $1-\beta$ appears here, in contrast to Lemma~\ref{lem:marginal_coverage}: the confidence level enters only through the radius $r_K(\beta)$, and once that radius is fixed \eqref{eq:dro_full_ball_coverage} holds surely rather than with probability $1-\beta$.
  \resultend
\end{remark}

So far, we have presented CP- and DRO-based estimators of $\bar{\alpha}$ in \eqref{eq:cal_cond_goal}, with their connections and key differences summarized in Table~\ref{tab:comparison_without_shift}. Both inflate the empirical quantile to hedge against the gap between $\widehat{\bbP}_K$ and $\bbP$, but along different axes: CP raises the probability level at which the empirical quantile is evaluated, from the nominal $1-\delta$ to $1-\delta+\epsilon$, where the level inflation $\epsilon>0$ is the finite-sample correction supplied by Lemmas~\ref{lem:2}--\ref{lem:4}; DRO instead enlarges the quantile value through a Wasserstein ambiguity radius. We refer to these as the \emph{level} and \emph{value} coordinates of the correction, the first living on the probability axis and the second in the units of the score. Their relative behavior depends on the tail of the underlying distribution, which we make precise as follows.

\begin{definition}[Light- and heavy-tailed distributions]
  \label{def:tails}
  A distribution $\mathbb{P}$ is \emph{light-tailed} if its right tail decays at least
  exponentially, i.e., there exist constants $C,a>0$ such that $\mathbb{P}(R>t)\le C e^{-at}$
  for all $t\ge 0$; otherwise $\mathbb{P}$ is \emph{heavy-tailed}.
\end{definition}

Figure~\ref{fig:cp_dro_intuition} shows how the tail of the score distribution governs the two corrections. Since CP evaluates an upper-tail order statistic, \(\epsilon\) barely moves the threshold for a light-tailed score, whose upper samples are dense near the quantile, but can push it much farther for a heavy-tailed score, whose upper samples are sparse~\cite{cleaveland2024conformal}. DRO instead adds a radius \(r\) directly in the score's units, determined by \(K\), \(\beta\), and the ambiguity-set construction rather than by the realized sample spacing. Sections~\ref{sec:examples} and~\ref{sec:exp_setup_cp_cls} evaluate both estimators on this basis. Calibration-conditional CP reliably attains the two-level guarantee~\eqref{eq:cal_cond_goal} with compact prediction sets and is our default. Its main weakness is heavy-tailed scores at small \(K\), where the level inflation may substantially overshoot the true quantile; in the extreme case, the selected threshold is \(\infty\), yielding a vacuous set. In this regime, a tuned DRO radius can be more accurate, although for bounded discrete scores the resulting DRO set may also become vacuous.

The two corrections are related: for a fixed calibration sample, both raise the threshold above the empirical $(1-\delta)$-quantile, CP in discrete steps through the upper order statistics and DRO continuously by $r$. They differ in what the data can \emph{certify}. CP's inflation follows from distribution-free concentration of the empirical CDF, via the Dvoretzky--Kiefer--Wolfowitz (DKW) or Hoeffding inequality, so the closed form $\epsilon=\sqrt{\ln(1/\beta)/(2K)}$ certifies~\eqref{eq:cal_cond_goal} for every $\bbP$. Achieving the same margin in value space costs a factor of the inverse density: the quantile function is the inverse of the CDF, so its slope at level $p$ is $1/f$, and a level margin $\epsilon$ therefore corresponds to a value margin $\epsilon/f$. Because $W_\infty$ requires the bound to hold across the whole support, the certified radius must use the uniform lower bound $m\le f$, giving $r_K(\beta)=\tfrac{1}{m}\sqrt{\ln(2/\beta)/(2K)}$ (Lemma~\ref{lem:choice-r}) and making the DRO shift the provably larger of the two (Section~\ref{subsec:stat_props}). The bound $m$ is usually unknown, so in practice $r$ is tuned and the guarantee is lost. In return, DRO certifies coverage uniformly over the ambiguity ball (Table~\ref{tab:comparison_without_shift}), at the price of conservativeness. CP is therefore the default for distribution-free coverage on i.i.d.\ data, while DRO returns a finite threshold at small $K$, where CP returns
$\infty$, and extends to test-time shift (Section~\ref{subsec:distribution_shift}).

%
%
% \marginJC{Shouldn't we comment here on how/when/under which conditions one approach will result in a better estimate than the other? And the sensibleness of the conditions? Right now, we present parallel results for both approaches, but beyond noting differences (adjusting index vs value), we don't quite exploit what we've derived to provide more insight.}
%

\subsection{Statistical Properties of CP and DRO Estimators}
\label{subsec:stat_props}

Beyond coverage guarantees, it is important to understand the statistical behavior of the CP and DRO estimators, including consistency, finite-sample distributions, and asymptotic behavior because these properties guide how conservative or efficient each method is in practice. We first establish the statistical properties of the CP quantile estimator introduced in Lemma~\ref{lem:2}, whose level correction $\sqrt{\ln(1/\beta)/(2K)}$ is available in closed form, whereas the level $\delta'$ of Lemma~\ref{lem:3} is defined implicitly by a beta-function condition. Analogous results hold for the variants in \eqref{eq: cali_cp_lemma3}.

\begin{lemma}\longthmtitle{Statistical Properties of CP Quantile Estimation}\label{lem:statistical-properties-cp-estimation}
    Let $\bar{\alpha}_{\mathrm{CP}}$ be defined as in Lemma~\ref{lem:2}, let $F$ denote the
    cumulative distribution function of $R^{(0)}$, and let $\alpha$ denote the $(1-\delta)$
    quantile of $R^{(0)}$ as defined in Problem~\ref{prob: UQ}. Suppose $F$ is continuous and
    strictly increasing in a neighborhood of $\alpha$, then
    $\bar{\alpha}_{\mathrm{CP}}\to\alpha$ almost surely as $K\to\infty$.
    Additionally, for any $t\in\mathbb{R}$,
    \begin{align}\label{eq:distribution-alphabar-cp}
        \mathbb{P}^K\big( \bar{\alpha}_{\mathrm{CP}} \leq t \big)
        = \sum_{i=m_\mathrm{CP}}^{K} \binom{K}{i} F(t)^i \big(1-F(t)\big)^{K-i},
    \end{align}
    where $m_{\mathrm{CP}} := \ceil*{(K+1)\,p_K}$ and $p_K := 1-\delta+\sqrt{\ln(1/\beta)/(2K)}$. Furthermore, if $F$ is differentiable at $\alpha$ with density $f(\alpha) > 0$, then as $K\to\infty$:
    \begin{equation}\label{eq:convergence-in-distribution-cp}
        \sqrt{K}(\bar{\alpha}_{\mathrm{CP}}-\alpha) \xrightarrow{d} {}\frac{\sqrt{\delta (1-\delta)}}{f(\alpha) } N(0,1) 
        + \frac{1}{f(\alpha)}\sqrt{\frac{\ln(1/\beta)}{2}},
    \end{equation}
    where $\;\xrightarrow{d}$ denotes convergence in distribution and $N(0,1)$ the standard normal distribution. \resultend
    %and $f(\cdot)$ is the probability density function of $R^{(0)}$. 
\end{lemma}

\begin{proof}
    By~\cite[Theorem 1]{LJH-GL:11}, we have
    $\Quant_{1-\delta}(R^{1:K},\infty)$ converges to $\alpha$ with
    probability one, whereas, for any fixed $\epsilon\in(0,\delta)$,
    $\Quant_{1-\delta+\epsilon}(R^{1:K},\infty)$ converges to the
    $(1-\delta+\epsilon)$ quantile of $R$ with probability one. Since $p_K-(1-\delta)=\sqrt{\ln(1/\beta)/(2K)}\to0$, for every fixed
    $\epsilon>0$, there exists $K_\epsilon$ such that $1-\delta\leq p_K\leq1-\delta+\epsilon$ for all $K\geq K_\epsilon$.
    By the monotonicity of quantiles,
    \begin{equation*}
    \Quant_{1-\delta}(R^{1:K},\infty)
    \leq \bar{\alpha}_{\mathrm{CP}}
    \leq \Quant_{1-\delta+\epsilon}(R^{1:K},\infty).
    \end{equation*}
    Taking the limit as $K\to\infty$ and then letting $\epsilon\to0^+$, the continuity and strict
    monotonicity of $F$ in a neighborhood of $\alpha$ imply that:
    \[
        \bar{\alpha}_{\mathrm{CP}} \longrightarrow \alpha
        \qquad \text{almost surely}.
    \]
    On the other hand,~\eqref{eq:distribution-alphabar-cp} follows from
    \cite[Section 2.3.4]{RJS:80}.

    We now show~\eqref{eq:convergence-in-distribution-cp}. Since
    $p_K\to1-\delta<1$, we have $p_K\leq K/(K+1)$ for all sufficiently large $K$. Hence, $m_{\mathrm{CP}}=\ceil*{(K+1)p_K}\leq K$, so the appended point $\infty$ is eventually inactive. Let $\hat p_K:=m_{\mathrm{CP}}/K$. Since $\hat p_K-p_K=O(K^{-1})$ and $p_K$ approaches $1-\delta$ at the rate $K^{-1/2}$, the quantile inflation contributes a deterministic offset that persists in the limit. We isolate the two effects by writing
    \begin{align*}
    \sqrt{K}\,(\bar{\alpha}_{\mathrm{CP}}-\alpha)
    &= \underbrace{\sqrt{K}\,\bigl(\bar{\alpha}_{\mathrm{CP}}-F^{-1}(\hat p_K)\bigr)}
    _{\text{fluctuation}} \\
    &\quad+
    \underbrace{\sqrt{K}\,\bigl(F^{-1}(\hat p_K)-F^{-1}(1-\delta)\bigr)}
    _{\text{drift}}.
    \end{align*}
    For the fluctuation term, $\bar{\alpha}_{\mathrm{CP}}$ is the empirical $\hat p_K$-quantile of $R^{1:K}$, where $\hat p_K\to1-\delta$. By the asymptotic normality of empirical quantiles at a converging level \cite[Section 2.3.3]{RJS:80},
    \begin{align*}
    \sqrt{K}\,\bigl(
    \bar{\alpha}_{\mathrm{CP}}-F^{-1}(\hat p_K)
    \bigr)
    \xrightarrow{d}
    \frac{\sqrt{\delta(1-\delta)}}{f(\alpha)}\,N(0,1).
    \end{align*}
    For the drift term, since $\hat p_K-p_K=O(K^{-1})$ and $p_K-(1-\delta)=\tfrac{1}{\sqrt{K}}\sqrt{\ln(1/\beta)/2}$, we have
    \begin{align*}
    \sqrt{K}\bigl(\hat p_K-(1-\delta)\bigr)
    \to \sqrt{\frac{\ln(1/\beta)}{2}}.
    \end{align*}
    Moreover, $F^{-1}$ is differentiable at $1-\delta$ with derivative
    $1/f(\alpha)$. Therefore, a first-order expansion gives
    \begin{align*}
    \sqrt{K}\bigl(F^{-1}(\hat p_K)-F^{-1}(1-\delta)\bigr)
    \to \frac{1}{f(\alpha)}
    \sqrt{\frac{\ln(1/\beta)}{2}}.
    \end{align*}
    Since the drift converges to a constant, Slutsky's theorem~\cite[Section 1.5.4]{RJS:80} yields~\eqref{eq:convergence-in-distribution-cp}.
\end{proof}

Next, we establish the statistical properties of the DRO quantile estimator introduced in Lemma~\ref{lem:dro_quantile_refined}.

\begin{lemma}\longthmtitle{Statistical Properties of DRO Quantile Estimation}\label{lem:statistical-properties-dro-quantile-estimation}
    Let $\bar{\alpha}_{\mathrm{DRO}}$ be defined as in Lemma~\ref{lem:dro_quantile_refined},
    let $F$ denote the cumulative distribution function of $R^{(0)}$, and let $\alpha$ be its
    $(1-\delta)$ quantile. Suppose that $F$ is continuous and strictly increasing in a neighborhood of $\alpha$ and that $r_K(\beta)\to 0$ as $K\to\infty$. Then
    $\bar{\alpha}_{\mathrm{DRO}}\to\alpha$ almost surely as $K\to\infty$.
    Additionally, for any $t\in\mathbb{R}$,
    \begin{align}\label{eq:distribution-alphabar-dro}
        \mathbb{P}^K\big( \bar{\alpha}_{\mathrm{DRO}} \leq t \big)
        = \sum_{i=m_{\mathrm{DRO}}}^{K} \binom{K}{i} F(u)^i \big(1-F(u)\big)^{K-i},
    \end{align}
    where $u := t-r_K(\beta)$ and $m_{\mathrm{DRO}} := \ceil*{K(1-\delta)}$. Furthermore, suppose that $F$ is differentiable
    at $\alpha$ with density $f(\alpha) > 0$. If $\sqrt{K}\,r_K(\beta)\to c_\beta$ for some constant $c_\beta\geq0$, then 
    \begin{equation} 
    \label{eq:dro_asympto_general} 
    \sqrt{K}\left( \bar{\alpha}_{\mathrm{DRO}}-\alpha \right) \xrightarrow{d} \frac{\sqrt{\delta(1-\delta)}}{f(\alpha)}N(0,1) +c_\beta. 
    \end{equation}
    as $K\to\infty$. This condition holds, for example, for the choice of $r_K(\beta)$ in Lemma~\ref{lem:choice-r}. \resultend
\end{lemma}

The identity in \eqref{eq:distribution-alphabar-dro} follows from the distribution of order statistics~\cite[Section 2.3.4]{RJS:80}. The strong consistency follows from the consistency of empirical quantiles \cite[Theorem 1]{LJH-GL:11} together with $r_K(\beta)\to0$, while \eqref{eq:dro_asympto_general} follows from the asymptotic normality of empirical quantiles~\cite[Section 2.3.3]{RJS:80} and Slutsky's theorem~\cite[Section 1.5.4]{RJS:80}.

Lemmas~\ref{lem:statistical-properties-cp-estimation}  and~\ref{lem:statistical-properties-dro-quantile-estimation}  show that both \(\bar{\alpha}_{\mathrm{CP}}\) and \(\bar{\alpha}_{\mathrm{DRO}}\) are strongly consistent estimators of \(\alpha\). Moreover, their estimation errors are
asymptotically of order \(K^{-1/2}\) and decompose into a random sampling fluctuation and a deterministic upward correction. The random component is identical for the two
estimators and converges, after multiplication by \(\sqrt{K}\), to
\(\frac{\sqrt{\delta(1-\delta)}}{f(\alpha)}N(0,1).\)

For CP, the deterministic correction arises from inflating the empirical quantile level above \(1-\delta\), resulting in the asymptotic shift
\(\frac{1}{f(\alpha)}\sqrt{\frac{\ln(1/\beta)}{2}}\). For DRO, the correction is induced by the Wasserstein radius
\(r_K(\beta)\), and its asymptotic contribution is \(c_\beta\) whenever \(\sqrt{K}r_K(\beta)\to c_\beta\). 

Thus, the two estimators share the same first-order asymptotic variance and differ in a deterministic shift. The shifts, however, are governed by the score distribution differently: CP's is set by the density at the target quantile, whereas DRO's is inherited from the ambiguity set radius. For Lemma~\ref{lem:choice-r}, where \(\sqrt{K}\,r_K(\beta)=\frac{1}{m}\sqrt{\ln(2/\beta)/2}\) and \(m\) lower bounds the density on the whole support,
\begin{align*}
    \underbrace{\frac{1}{f(\alpha)}\sqrt{\frac{\ln(1/\beta)}{2}}}_{\text{CP shift}}
    \;<\;
    \underbrace{\frac{1}{m}\sqrt{\frac{\ln(2/\beta)}{2}}}_{\text{DRO shift}},
\end{align*}
since \(m\le f(\alpha)\). Thus, CP pays a \emph{local} price and DRO a \emph{uniform} one and, with this radius, the DRO correction is never the smaller of the two. The gap widens as the density varies more sharply over the support, which can be observed in Figure~\ref{fig:cp_dro_intuition}. As the caption notes, the value-space margin shown is illustrative, and realizing it requires a radius tighter than \eqref{eq:wasserstein_inf_radius} (Section~\ref{sec:choosing_radius}).

\subsection{Choosing the Wasserstein Radius}
\label{sec:choosing_radius}

The choice of the Wasserstein radius is central to the statistical interpretation and practical performance of DRO. Lemma~\ref{lem:choice-r} provides one sufficient choice of $r_K(\beta)$ such that:
\begin{equation}
\bbP^K\left(
\bbP\in\mathcal{M}^{r_K(\beta)}(\widehat{\bbP}_K)
\right)\ge 1-\beta.
\end{equation}
More generally, finite-sample Wasserstein radii can be constructed using measure-concentration or statistical-inference arguments \cite{fournier2015rate,gao2023finite,blanchet2019robust}. Such choices give the ambiguity set the interpretation of a high-confidence region for the unknown distribution. However, the resulting bounds may depend on unknown distributional constants and can be overly conservative in finite samples.

Alternatively, the radius may be selected by validation or treated as a design parameter controlling the trade-off between robustness and conservativeness \cite{cisneros2020distributionally, long2024sensorbased_dro, long2023dro_lf}, or calibrated around a fitted center to cover the unknown distribution~\cite{guo2026learning}. For any fixed
radius $r$, the threshold \(\Quant_{1-\delta}(R^{1:K})+r\) still guarantees coverage uniformly over $\mathcal{M}^{r}(\widehat{\bbP}_K)$. However, if $r$ is selected heuristically, the high-probability coverage guarantee for the true distribution $\bbP$ does not follow. 

CP provides another useful reference for the scale of the radius. Whenever the CP estimator $\bar\alpha_{\mathrm{CP}}$ is finite, define its effective value-space correction as
\begin{equation}
\label{eq:cp_effective_radius}
r_{\mathrm{CP}}(R^{1:K})
:=
\bar\alpha_{\mathrm{CP}}
-
\Quant_{1-\delta}(R^{1:K}).
\end{equation}
By construction,
\begin{equation}
\Quant_{1-\delta}(R^{1:K})
+r_{\mathrm{CP}}(R^{1:K})
=
\bar\alpha_{\mathrm{CP}}.
\end{equation}
Thus, $r_{\mathrm{CP}}$ translates CP's quantile-level inflation into an equivalent value-space correction for the realized calibration sample. It may therefore serve as a reference or initialization when selecting a practical DRO radius. Importantly, it is not generally a certified Wasserstein radius: CP controls the target quantile, but does not guarantee that the distribution $\bbP$ lies in $\mathcal{M}^{r_{\mathrm{CP}}}(\widehat{\bbP}_K)$ with high probability.

To interpret the scale of this correction, let $\epsilon_K:=\sqrt{\ln(1/\beta)/(2K)}$, and suppose that $\bbP$ has a
continuous density $f$ that is positive near its $(1-\delta)$-quantile $\alpha$. A first-order quantile approximation gives:
\begin{equation} \label{eq:cp_radius_local_approximation}
r_{\mathrm{CP}} \approx \frac{\epsilon_K}{f(\alpha)}.
\end{equation}
Hence, the conversion from CP's level correction to a value-space radius depends on the local density near the target quantile. For the special case $\bbP=\mathrm{Uniform}[0,1]$, where $f(\alpha)=1$, this reduces to $r_{\mathrm{CP}}\approx\epsilon_K$. For general distributions, no distribution-free constant conversion exists.

\subsection{Illustrative Examples}
\label{sec:examples}

We now use simple scalar distributions to illustrate the finite-sample mechanisms identified in the preceding analysis. The examples visualize how sample size, the shape of the score distribution, and the choice of Wasserstein radius affect the CP and DRO quantile estimators. Their purpose is to isolate and explain these theoretical effects in controlled settings, rather than to evaluate performance on a particular application. The studies in Section~\ref{sec:exp_setup_cp_cls}, instead, examine the same methods on more realistic problems.

We consider distributions with bounded, unbounded, and truncated
supports. Throughout, we set $\delta=0.1$ and $\beta=0.05$. For each
distribution and sample size $K$, we repeat the calibration procedure
over 500 independent trials, drawing a new calibration set in each trial, and report the mean and standard deviation of the
resulting quantile estimates.
%
% \marginJC{We explain here what we do, but we do not explain the goal. Why do we do it? We want to illustrate the results? Visualize the insights? Provide visual representations of the things we've unveiled? All of the above? Plus, the section after this is called "Simulations". So explain to the reader why the 2 separate sections -- I guess each one has a different purpose.}
%

\textbf{Sample size and conservativeness.}
%
% \marginJC{If I was a reader, I'd be unhappy about this title. Superdifficult to parse, almost impossible to understand what we mean, ugly to have Lemma~BLA in a title that, in general, should be self-explanatory. Worst of all, we have given no indication to the reader as to what is the issue that we're trying to analyze: I guess the choice of radius in DRO, the impact it has on the uncertanty quantification, the options opened by looking at it beyond the sufficient condition in Lemma IV.5,...}
%
Figure~\ref{fig:uniform_guarantee} compares the three CP corrections with the
DRO estimator of Lemma~\ref{lem:dro_quantile_refined} on
$\mathbb{P} = \mathrm{Uniform}[0,1]$, for which $\inf f = 1$ and the true
$(1-\delta)$-quantile is $0.9$. Panel~(a) shows the non-vacuity thresholds of
Table~\ref{tab:kmin}: Lemma~\ref{lem:3} becomes finite at $K=29$, whereas
Lemmas~\ref{lem:2} and~\ref{lem:4} require $K=170$ and $K=172$. The DRO estimator is finite at every $K$, but its certified radius
$r_K(\beta) = \sqrt{\ln(2/\beta)/(2K)}$ equals $\approx 1.358$ at $K=1$, so the
bound it returns is loose until $K$ reaches the hundreds. Panel~(b) confirms
the ordering predicted by the asymptotic shifts: once all four are available,
the certified DRO threshold is the most conservative and Lemma~\ref{lem:3} the
tightest.

\begin{figure*}[htb]
    \centering
    \includegraphics[width=\linewidth]{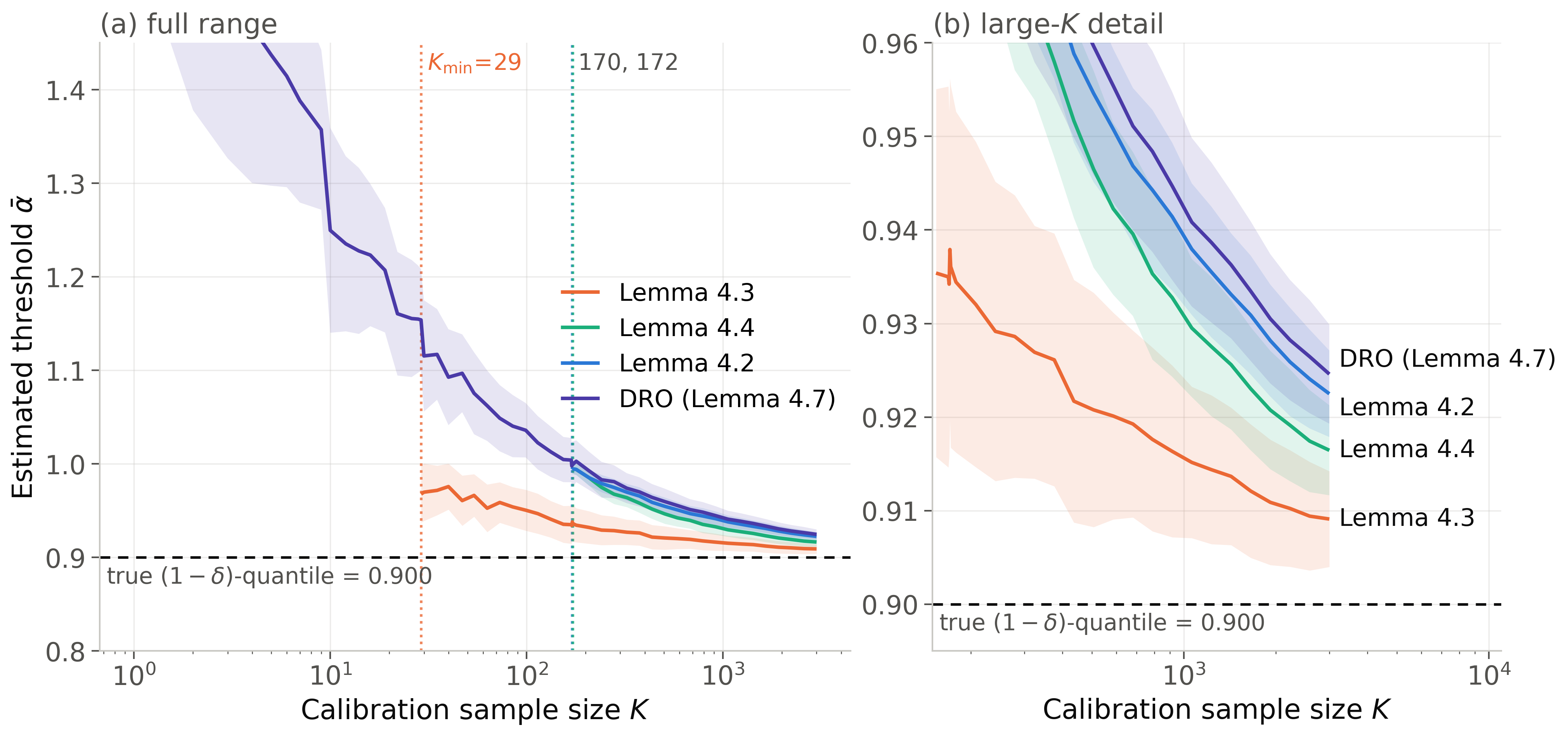}
    \caption{The three CP corrections and the certified-radius DRO
    estimator on $R \sim \mathrm{Uniform}[0,1]$ ($\delta=0.1$, $\beta=0.05$;
    mean and one standard deviation over $500$ trials).
    \textbf{(a)} Full range. Dotted verticals mark the non-vacuity thresholds
    $K_{\min}=29$ for Lemma~\ref{lem:3} and $170$, $172$ for
    Lemmas~\ref{lem:2} and~\ref{lem:4}; below these each estimator returns
    $\infty$ and is omitted. DRO is finite at every $K$, at the price of a
    large certified radius when $K$ is small.
    \textbf{(b)} Large-$K$ detail: the four estimates appear in the order Lemma~\ref{lem:3} $<$ Lemma~\ref{lem:4} $<$ Lemma~\ref{lem:2} $<$ DRO, as the asymptotic shifts predict.}
    \label{fig:uniform_guarantee}
    \vspace{-1ex}
\end{figure*}

\textbf{Effect of the DRO radius.}
We next compare CP with DRO using the simplified scaling
$r_K = r_1/\sqrt{K}$,
%
% \marginJC{This is slightly loose, wouldn't mind if we were more explicit about it. I mean, we're saying: we replace the def of $r_K(\beta)$ in (13) by this $r_K = r_1/\sqrt{K}$. I'd flat out specify it.}
%
which preserves the asymptotic $K^{-1/2}$ rate of
Lemma~\ref{lem:choice-r} while letting $r_1$ act as a tunable parameter. Figure~\ref{fig:multi_radius_panel} shows the
comparison on three distributions: $\mathrm{Uniform}[0,1]$, standard
Normal $\mathcal{N}(0,1)$, and
$\mathrm{Beta}(2,2)$. Across all $r_1 \in \{0.1, 0.5, 1, 4\}$, DRO converges to the true quantile, with larger $r_1$ giving more
conservative estimates at small $K$. The Normal needs the largest $r_1$: its density at the target quantile is the smallest of the three, so a level correction buys the least threshold there. For small $r_1$, the empirical $(1-\delta)$-quantile is often
drawn from samples that miss the upper tail, and the DRO estimate sits
below the truth before converging.

\begin{figure*}[htb]
    \centering
    \includegraphics[width=\linewidth]{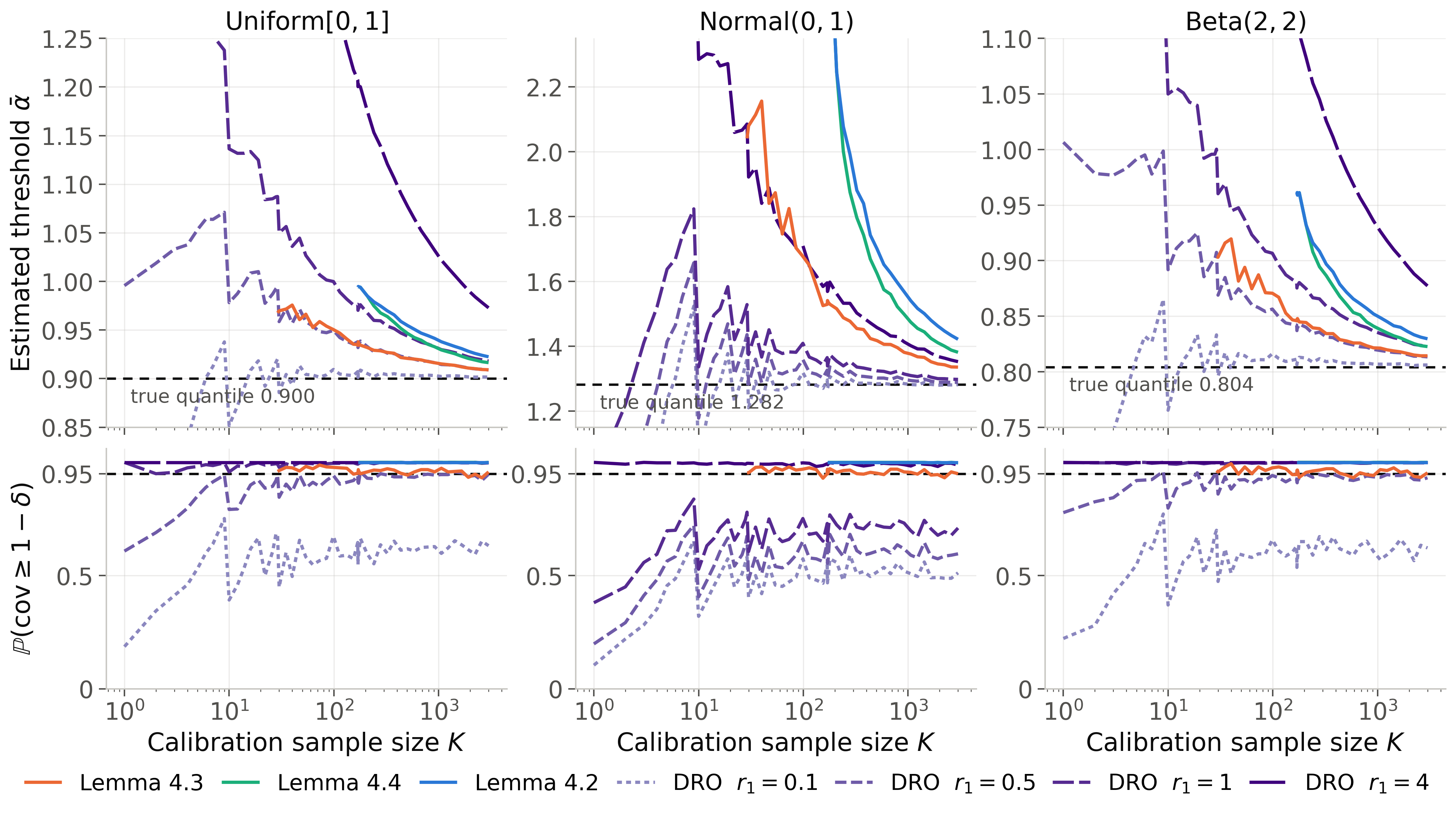}
    \caption{The three CP corrections against DRO with the simplified
    radius $r_K = r_1/\sqrt{K}$, on $\mathrm{Uniform}[0,1]$ (left), standard
    Normal (center), and $\mathrm{Beta}(2,2)$ (right) at $\delta=0.1$,
    $\beta=0.05$. DRO radii use a single-hue ramp, darker for larger $r_1$.
    \textbf{Top:} mean threshold over $500$ trials. Each CP curve begins at its
    non-vacuity threshold, and all methods converge to the true
    $(1-\delta)$-quantile (dashed) as $K$ grows.
    \textbf{Bottom:} the two-level success rate, which must lie on or above
    $1-\beta$. Lemmas~\ref{lem:2} and~\ref{lem:4} exceed the target by a wide
    margin, attaining a rate of one, whereas Lemma~\ref{lem:3} attains it with no
    slack, so its empirical rate fluctuates about $1-\beta$ within Monte Carlo
    error. DRO meets the
    target only for a large enough radius: $r_1\ge 1$ suffices on Uniform and Beta, whereas the Normal requires $r_1=4$, whose density at the target quantile is far smaller. Convergence of the estimate does not by itself imply
    the guarantee.}
    \label{fig:multi_radius_panel}
    \vspace{-2ex}
\end{figure*}

\textbf{Quantitative convergence.}
Table~\ref{tab:convergence} reports mean estimates of
$\bar\alpha_\mathrm{CP}$ and $\bar\alpha_\mathrm{DRO}$ (each at a radius whose two-level rate reaches $1-\beta$)
at
five sizes from $K=1$ to $10^4$, across five distributions including
Exponential (unbounded) and a Truncated Normal on $[-1,1]$, over 500
independent runs. Two observations emerge:
\emph{(i)} DRO yields a finite, usable estimate at $K = 1$ for all five
distributions, whereas CP is vacuous;
\emph{(ii)} at $K = 1000$, CP overshoots the true quantile by $4$--$22\%$,
with the largest overshoots on heavier-tailed supports (Exp$(1)$: $22\%$;
Normal: $21\%$; Truncated Normal: $12\%$; Uniform: $4\%$; Beta: $6\%$).
DRO, at radii that meet the two-level guarantee, stays within $3$--$10\%$ on the same five distributions.
The disparity reflects CP's $\sqrt{\ln(1/\beta)/(2K)}$ level correction,
which shifts quantile selection deep into the upper tail when the density
is small there; DRO's additive radius is symmetric.

\textbf{Discussion.}
With an appropriately chosen Wasserstein radius, DRO is competitive
with CP and, on heavier tails, tighter than the relaxation-based corrections
of Lemmas~\ref{lem:2} and~\ref{lem:4}. The advantage is most pronounced in
two regimes: low $K$, where CP is vacuous; and heavier tails, where CP's level
inflation overshoots. The optimal $r_1$ tracks the density at the
$(1-\delta)$-quantile: compactly supported distributions whose density stays
bounded away from zero there, such as $\mathrm{Uniform}[0,1]$ and
$\mathrm{Beta}(2,2)$, tolerate small $r_1$, whereas the Normal, whose density
at the quantile is far smaller, requires a larger one.

However, this comparison is not symmetric in guarantees. CP enjoys a 
distribution-free finite-sample coverage bound under only the iid 
assumption, via the DKW level correction $\sqrt{\ln(1/\beta)/(2K)}$. 
DRO requires the chosen radius to upper-bound 
$W_\infty(\widehat{\bbP}_K, \mathbb{P})$ with probability $\geq 1 - \beta$. 
Lemma~\ref{lem:choice-r} certifies this for bounded densities with 
$\inf f \geq m > 0$, which among our five distributions the Uniform and the truncated Normal
satisfy; $\mathrm{Beta}(2,2)$ has
compact support but a density vanishing at the endpoints, while the Normal and
Exponential are unbounded. For the others, the simplified scaling 
$r_K = r_1/\sqrt{K}$ inherits the asymptotic rate of 
Lemma~\ref{lem:choice-r}, but $r_1$ becomes a heuristic tuning parameter. 
The empirical performance in Table~\ref{tab:convergence} is therefore 
evidence of practical utility, not a finite-sample guarantee.

\begin{table}[htb]
    \centering
    \small
    \caption{Mean $\hat\alpha_\mathrm{CP}$ and $\hat\alpha_\mathrm{DRO}$ (each at a radius whose two-level rate reaches $1-\beta$ at every $K$ shown, given per row) across five distributions and calibration sizes $K$,
    averaged over $500$ independent runs at $\delta = 0.1$, $\beta = 0.05$.
    Last column reports the true $(1-\delta)$-quantile for reference.
    Here CP is the estimator of Lemma~\ref{lem:2}, by Remark~\ref{rem:cp_finite_sample} it is vacuous for $K<170$, which is why the first three columns are $\infty$. At $K = 1000$, CP overshoots the truth by $4$--$22\%$, with the largest overshoots on heavier-tailed
    supports (Exp(1): $22\%$; Normal: $21\%$; Trunc.\ Normal: $12\%$), while DRO, at radii that meet the two-level guarantee, stays within $3$--$10\%$.}
    \label{tab:convergence}    
    \setlength{\tabcolsep}{4pt}
    \resizebox{\linewidth}{!}{\begin{tabular}{@{}lrrrrrc@{}}
    \toprule
                              & \multicolumn{5}{c}{Calibration size $K$} & \\
    \cmidrule(lr){2-6}
    Distribution, method      & 1        & 10       & 100      & 1000  & 10000 & Truth \\
    \midrule
    Uniform, CP                & $\infty$ & $\infty$ & $\infty$ & 0.939 & 0.912 & \multirow{2}{*}{0.900} \\
    Uniform, DRO ($r_1{=}1$) & 1.496 & 1.125 & 0.991 & 0.931 & 0.910 &       \\
    \midrule
    Normal, CP                 & $\infty$ & $\infty$ & $\infty$ & 1.550 & 1.356 & \multirow{2}{*}{1.282} \\
    Normal, DRO ($r_1{=}4$) & 3.987 & 2.247 & 1.648 & 1.404 & 1.322 &       \\
    \midrule
    Beta(2,2), CP              & $\infty$ & $\infty$ & $\infty$ & 0.851 & 0.817 & \multirow{2}{*}{0.804} \\
    Beta(2,2), DRO ($r_1{=}1$) & 1.506 & 1.056 & 0.897 & 0.835 & 0.814 &       \\
    \midrule
    Exp(1), CP                 & $\infty$ & $\infty$ & $\infty$ & 2.800 & 2.433 & \multirow{2}{*}{2.303} \\
    Exp(1), DRO ($r_1{=}6$) & 7.001 & 3.750 & 2.860 & 2.483 & 2.361 &       \\
    \midrule
    Trunc.\ N., CP             & $\infty$ & $\infty$ & $\infty$ & 0.841 & 0.777 & \multirow{2}{*}{0.749} \\
    Trunc.\ N., DRO ($r_1{=}2$) & 1.991 & 1.203 & 0.931 & 0.810 & 0.769 &       \\
    \bottomrule
\end{tabular}}
\end{table}

\section{CP and DRO With Distribution Shift}
\label{subsec:distribution_shift}

In this section, we present how to solve Problem~\ref{prob:UQ_test_shift} with CP and DRO. Motivated by the calibration-conditional formulation in \eqref{eq:cal_cond_goal}, we seek a threshold $\bar{\alpha}=\bar{\alpha}(R^{1:K})$ that satisfies
\begin{equation}
\label{eq:cal_test_cond_goal}
\bbP^K\!\left(
  \bbP_{\mathrm{test}}\!\big(R^{(0)} \le \bar{\alpha}(R^{1:K})\big) \;\ge\; 1-\delta
\right) \;\ge\; 1-\beta ,
\end{equation}
for a user chosen confidence level $\beta \in (0,1)$. 

In this case, one must additionally account for the discrepancy between $\bbP$ and $\bbP_{\mathrm{test}}$.
%
% \marginJC{The wording makes it sound as in the previous section, we had not accounted for the the discrepancy between the empirical distribution $\widehat{\bbP}_K$
% and $\bbP$. But we have, no?}
%
To make the goal in Problem~\ref{prob:UQ_test_shift} attainable, this discrepancy must be
bounded. We consider two shift models, each an instance of Problem~\ref{prob: UQ} with a
different pair $(\dist,\gamma)$: a Wasserstein model, which bounds how far mass moves, and a
L\'evy--Prokhorov model, which additionally allows a fraction of the mass to move arbitrarily.

%
% \marginJC{Comment on how reasonable the assumption is (seems very). I guess knowing the value of $r$ is the part that might be trickier to justify.}
%

% Assumption~\ref{assump: test_shift_wasserstein} provides an interpretable way to quantify the discrepancy between the calibration distribution $\mathbb{P}$ and the test distribution $\mathbb{P}_{\mathrm{test}}$ \NA{This sentence is repetitive and can be removed (we already said that before and in the assumption).}. 

The radius $\eta$ controls the amount of calibration-test shift allowed by the model and leads to the additive threshold correction used below. In practice,~$\eta$ can be specified from prior knowledge, estimated from additional validation data when available, or treated as a robustness parameter to study the trade-off between validity and conservativeness.

\subsection{Wasserstein Shift Model}

\begin{Assumption}
\label{assump: test_shift_wasserstein}
The test distribution is within a $W_\infty$-ball of radius $\eta>0$ around the calibration distribution:
\begin{equation}
\label{eq: test_shift_wasserstein}
   W_\infty(\bbP,\bbP_{\mathrm{test}})\le \eta,
\end{equation}
corresponding to Problem~\ref{prob: UQ} with $\dist=W_\infty$ and $\gamma=\eta$.
\end{Assumption}

\textbf{Conformal prediction.} Under Assumption~\ref{assump: test_shift_wasserstein}, the shift is
absorbed by adding the Wasserstein radius $\eta$ to the calibration-conditional CP threshold.

\begin{lemma}[\textbf{Shifted robust CP estimator}]
\label{lem:cp_shift}
Fix $\delta,\beta\in(0,1)$. Given $R^{(1)},\ldots,R^{(K)} \sim \bbP$ i.i.d.\ and $R^{(0)}\sim \bbP_{\mathrm{test}}$ satisfying Assumption~\ref{assump: test_shift_wasserstein}, define the CP estimator $\bar{\alpha}_{\mathrm{CP}}$ as in \eqref{eq:cp_cond}, \eqref{eq: cali_cp_lemma2}, or \eqref{eq: cali_cp_lemma3}. The  test time estimator
\begin{equation}
\label{eq:cp_test_estimator}
\bar{\alpha}_{\mathrm{CP}}^{\mathrm{test}}
:= 
\bar{\alpha}_{\mathrm{CP}} + \eta ,
\end{equation}
satisfies the calibration-conditional guarantee:
\begin{equation}
\label{eq:cp_test_cond}
\bbP^{K}\!\left(
  \bbP_{\mathrm{test}}\!\big(R^{(0)} \le \bar{\alpha}_{\mathrm{CP}}^{\mathrm{test}}\big) \ge 1-\delta
\right) \ge 1-\beta . 
\end{equation} 
\resultend
\end{lemma}

\begin{proof}
By Lemma~\ref{lem:2} (or its variants \eqref{eq: cali_cp_lemma2}, \eqref{eq: cali_cp_lemma3}), we have
\begin{equation}
\label{eq:cp_base_cond}
\bbP^{K}\!\left(
  \bbP \big(X \le \bar{\alpha}_{\mathrm{CP}}\big) \ge 1-\delta
\right) \ge 1-\beta,
\end{equation}
where $X\sim\bbP$. Under Assumption~\ref{assump: test_shift_wasserstein}, $W_{\infty}(\bbP,\bbP_{\mathrm{test}})\le \eta$, so there exists a coupling $\pi\in\Gamma(\bbP,\bbP_{\mathrm{test}})$ such that for $(X,Y)\sim\pi$ we have $|X-Y|\le \eta$ almost surely.
Consequently, for any $t\in\mathbb{R}$,
\begin{equation}
\label{eq:cdf_shift_winf}
\bbP_{\mathrm{test}}(Y \le t+\eta)
~\ge~
\bbP(X \le t).
\end{equation}
Applying \eqref{eq:cdf_shift_winf} with $t=\bar{\alpha}_{\mathrm{CP}}$ yields
\begin{equation}
\label{eq:cp_shift_transfer}
\bbP_{\mathrm{test}}\big(Y \le \bar{\alpha}_{\mathrm{CP}}+\eta\big)
~\ge~
\bbP\big(X \le \bar{\alpha}_{\mathrm{CP}}\big).
\end{equation}
Therefore, on the event
$\{\bbP(X \le \bar{\alpha}_{\mathrm{CP}})\ge 1-\delta\}$, we also have
$\bbP_{\mathrm{test}}(Y \le \bar{\alpha}_{\mathrm{CP}}^{\mathrm{test}})
\ge 1-\delta$ by \eqref{eq:cp_shift_transfer}. Combining this implication with
\eqref{eq:cp_base_cond} proves \eqref{eq:cp_test_cond}.
\end{proof}

Lemma~\ref{lem:cp_shift} has a simple interpretation. The conformal threshold $\bar{\alpha}_{\mathrm{CP}}$ controls finite-sample uncertainty under the calibration distribution $\bbP$, while Assumption~\ref{assump: test_shift_wasserstein} accounts for the additional discrepancy between
$\bbP$ and $\bbP_{\mathrm{test}}$. Under the $W_\infty$ shift model, this second source of uncertainty enters as an additive correction in value space. Thus, test-time robustness is obtained by increasing the calibration threshold by $\eta$, at the cost of additional conservativeness.

%
% \marginJC{It's ugly to have a dry exposition: result, proof, result, proof. Better to provide insights before moving on.}
%

\textbf{Distributionally robust optimization.} Next, we augment the DRO estimator by the same additive $\eta$, combining the empirical-calibration gap and the calibration-test shift.

\begin{lemma}[\textbf{Shift robust DRO estimator}]\label{lem:dro_shift}
Fix $\delta,\beta\in(0,1)$. Given $R^{(1)},\ldots,R^{(K)} \sim \bbP$ i.i.d.\ and $R^{(0)}\sim \bbP_{\mathrm{test}}$ satisfying Assumption~\ref{assump: test_shift_wasserstein}, 
% Let $r_K(\beta)$ be chosen so that $\bbP^{K}\!\big(W_\infty(\bbP,\widehat{\bbP}_K)\le r_K(\beta)\big)\ge 1-\beta$. 
define the DRO estimator $\bar{\alpha}_{\mathrm{DRO}}$ as in \eqref{eq: alpha_dro_bound}. The test time estimator
\begin{equation}
\label{eq:dro_test_estimator}
\bar{\alpha}_{\mathrm{DRO}}^{\mathrm{test}}
:=
\bar{\alpha}_{\mathrm{DRO}} + \eta ,
\end{equation}
where $r_K(\beta)$ is defined as in Lemma~\ref{lem:choice-r},
satisfies the two-level guarantee
\begin{equation}
\label{eq:dro_test_cond}
\bbP^{K}\!\left(
  \bbP_{\mathrm{test}}\!\big(R^{(0)} \le \bar{\alpha}_{\mathrm{DRO}}^{\mathrm{test}}\big) \ge 1-\delta
\right) \ge 1-\beta .
\end{equation}
\resultend
\end{lemma}
%
% \marginJC{We're not consistent with the use of the symbol \resultend at the end of results. Some times we have it, some times we don't}
% %

The proof follows by combining Lemma~\ref{lem:dro_quantile_refined}
with the same $W_\infty$ shift argument used in
Lemma~\ref{lem:cp_shift}. Lemma~\ref{lem:dro_quantile_refined} gives
the calibration-distribution guarantee for
$\bar{\alpha}_{\mathrm{DRO}}$, while Assumption~\ref{assump:
test_shift_wasserstein} implies that adding $\eta$ transfers this guarantee
from $\bbP$ to $\bbP_{\mathrm{test}}$.
Thus, the DRO correction decomposes into two additive terms: $r_K(\beta)$ accounts for finite-sample uncertainty between the empirical and calibration distributions, while $\eta$ accounts for calibration-test distribution shift.

%
% \marginJC{Here we start with the LP stuff, which seems to combine elements from Wasserstein and CP. Should we present this as its own subsection?}
%

\subsection{L\'evy--Prokhorov Shift Model}

A complementary line of work~\cite{aolaritei2025conformal}
models test-time shift using the L\'evy--Prokhorov (LP) metric, which
allows both value-space perturbations and mass displacement. This leads to uncertainty thresholds that combine an additive value correction with a quantile-level correction. 

As a special case, we first introduce the total variation (TV) distance between two distributions $\mathbb{P}, \mathbb{Q}$:
\begin{equation}
\label{eq:tv_distance}
  \mathrm{TV}(\mathbb{P},\mathbb{Q}) 
  := \inf_{\pi \in \Gamma(\mathbb{P},\mathbb{Q})} 
    \int_{\mathcal{X}\times\mathcal{X}} \mathds{1}\{x_1 \neq x_2\} \, d\pi(x_1,x_2).
\end{equation}
Here $\mathds{1}$ denotes the indicator function. Intuitively, $\mathrm{TV}(\mathbb{P},\mathbb{Q})$ is the minimum
%
% \marginJC{maximal?}
%
fraction of mass that must be reassigned to transform $\mathbb{P}$ into $\mathbb{Q}$. 

The L{\'e}vy--Prokhorov (LP) distance generalizes the Wasserstein and TV distances by permitting both local and global perturbations. For $\epsilon \geq 0$, let
\begin{equation}\label{eq:lp_distance}
  \mathrm{LP}_\epsilon(\mathbb{P},\mathbb{Q}) := 
  \inf_{\pi \in \Gamma(\mathbb{P},\mathbb{Q})} 
  \int_{\mathcal{X}\times\mathcal{X}} 
    \mathds{1}\{\|x_1 - x_2\| > \epsilon\} \, d\pi(x_1,x_2),
\end{equation}
The corresponding LP ambiguity set is
\begin{equation}\label{eq: LP_ball}
  B_{\eta,\rho}(\mathbb{P}) := 
  \left\{ \mathbb{Q} \in \mathcal{P}(\mathcal{X}) \,\middle|\, \mathrm{LP}_{\eta}(\mathbb{P},\mathbb{Q}) \leq \rho \right\}.
\end{equation}
This admits a two-step interpretation: (i) each unit of mass under $\mathbb{P}$ can be moved within a radius $\eta$, and (ii) up to a fraction $\rho$ of the total mass may be displaced arbitrarily. Notably, $B_{0,\rho}(\mathbb{P})$ reduces to the TV ball, while $B_{\eta,0}(\mathbb{P})$ reduces to the $W_\infty$ ball.

The second shift model bounds the LP distance between the calibration and test distributions.

\begin{Assumption}
\label{assump:test_shift_LP}
The test distribution is within an LP ball around the calibration distribution: 
\begin{equation}\label{eq:lp_shift_ball}
\mathrm{LP}_{\eta}(\bbP,\bbP_{\mathrm{test}}) \le \rho ,
% \bbP_{\mathrm{test}} \in \mathcal{B}_{r,\rho}(\bbP) \;\coloneq\; \bigl\{\bbQ:\; \mathrm{LP}_r(\bbP,\bbQ)\le \rho \bigr\}.
\end{equation}
corresponding to Problem~\ref{prob: UQ} with $\dist=\mathrm{LP}_\eta$ and $\gamma=\rho$.
\end{Assumption}

% The cases $\rho=0$ and $\eta=0$ reduce to pure $W_\infty$–type and TV–type effects, respectively \NA{We already said this above Assumption 5.4.}.

Under Assumption~\ref{assump:test_shift_LP}, we restate the LP-robust conformal result of \cite[Cor.~4.2]{aolaritei2025conformal}.

\begin{theorem}
\label{thm:lp_robust_cp}
Fix $\delta \in (0,1)$. 
Given $R^{(1)},\ldots,R^{(K)} \sim \bbP$ i.i.d.\ and $R^{(0)} \sim \bbP_{\mathrm{test}}$ 
satisfying Assumption~\ref{assump:test_shift_LP}, 
define the LP robust conformal quantile as:
\begin{equation}
\label{eq:LP_estimator_explicit}
\bar\alpha_{\eta,\rho}(\widehat{\bbP}_K)
:= \Quant_{\,1-\delta+\rho+\tfrac{2+\rho-\delta}{K}}(R^{1:K},\infty) + \eta .
\end{equation}
Then, the following marginal guarantee holds:
%
% \marginJC{Why are we suddenly using curly brackets in these probability expressions?}
%
\begin{equation}
\label{eq:LP_single_prob}
(\bbP^{K}\times\bbP_{\mathrm{test}})\!\left(
R^{(0)} \le \bar\alpha_{\eta,\rho}(\widehat{\bbP}_K)
\right) \;\ge\; 1-\delta .
\end{equation}
\resultend
\end{theorem}

No non-vacuity condition is needed here: if the quantile level exceeds
$K/(K+1)$ the estimator returns $\infty$ and the guarantee holds trivially.
Remark~\ref{rem:cp_finite_sample} gives the condition under which the threshold
is finite, and hence useful.

The results of \cite{aolaritei2025conformal} provide single-probability 
(marginal) coverage guarantees under LP distribution shift. 
We now strengthen this to the calibration-conditional (two-level) guarantee of \eqref{eq:cal_cond_goal}, which holds with high probability over the calibration sample, by combining the usual CP index inflation in Lemma~\ref{lem:2} with the LP shift relation.

\begin{table*}[htb]
\centering
\small
\caption{Summary of uncertainty quantification under distribution shift.}
\label{tab:comparison_test_shift}
\setlength{\tabcolsep}{4pt}
\renewcommand{\arraystretch}{1.2}
\resizebox{\linewidth}{!}{
\begin{tabular}{@{}p{3.8cm}p{2.5cm}p{6cm}p{5.5cm}p{2.2cm}@{}}
\toprule
\textbf{Method} &
\textbf{Shift model} &
\textbf{Coverage guarantee} &
\textbf{Estimator $\displaystyle \bar\alpha(R^{1:K})$} &
\textbf{Result} \\
\midrule
% -------------------------------------------------------------------
\textbf{CP under $W_\infty$ shift} &
$\bbP_{\mathrm{test}}\in\calM^{\eta}(\bbP)$ &
$\displaystyle
\bbP^K\Bigl[
\bbP_{\mathrm{test}}\{R^{(0)}\le\bar\alpha\}\ge 1-\delta
\Bigr]\ge 1-\beta$ &
$\displaystyle
\Quant_{\,1-\delta+\sqrt{\frac{\ln(1/\beta)}{2K}}}
\bigl(R^{1:K},\infty\bigr)+\eta$ &
Lemma~\ref{lem:cp_shift} \\[8pt]
% -------------------------------------------------------------------
\textbf{DRO under $W_\infty$ shift} &
$\bbP_{\mathrm{test}}\in\calM^{\eta}(\bbP)$ &
$\displaystyle
\bbP^K\Bigl[
\bbP_{\mathrm{test}}\{R^{(0)}\le\bar\alpha\}\ge 1-\delta
\Bigr]\ge 1-\beta$ &
$\displaystyle
\Quant_{\,1-\delta}(R^{1:K})
+r_K(\beta)+\eta$ &
Lemma~\ref{lem:dro_shift} \\[8pt]
% -------------------------------------------------------------------
\textbf{LP robust CP (marginal) \cite{aolaritei2025conformal}} &
$\bbP_{\mathrm{test}}\in\mathcal{B}_{\eta,\rho}(\bbP)$ &
$\displaystyle
(\bbP^K\times\bbP_{\mathrm{test}})
\{R^{(0)}\le\bar\alpha\}\ge 1-\delta$
\; &
$\displaystyle
\Quant_{\,1-\delta+\rho+\frac{2+\rho-\delta}{K}}
\bigl(R^{1:K},\infty\bigr)+\eta$ &
Theorem~\ref{thm:lp_robust_cp} \\[8pt]
% -------------------------------------------------------------------
\textbf{CP under LP shift} &
$\bbP_{\mathrm{test}}\in\mathcal{B}_{\eta,\rho}(\bbP)$ &
$\displaystyle
\bbP^K\Bigl[
\bbP_{\mathrm{test}}\{R^{(0)}\le\bar\alpha\}\ge 1-\delta
\Bigr]\ge 1-\beta$ &
$\displaystyle
\Quant_{\,1-\delta+\sqrt{\frac{\ln(1/\beta)}{2K}}+\rho}
\bigl(R^{1:K},\infty\bigr)+\eta$ &
Theorem~\ref{thm:pac_levy} \\[8pt]
% -------------------------------------------------------------------
\textbf{DRO under LP shift} &
$\bbP_{\mathrm{test}}\in\mathcal{B}_{\eta,\rho}(\bbP)$ &
$\displaystyle
\bbP^K\Bigl[
\bbP_{\mathrm{test}}\{R^{(0)}\le\bar\alpha\}\ge 1-\delta
\Bigr]\ge 1-\beta$ &
$\displaystyle
\Quant_{\,1-\delta+\rho}(R^{1:K})
+r_K(\beta)+\eta$ &
Theorem~\ref{thm:lp_dro_shift} \\
\bottomrule
\end{tabular}}
\end{table*}

\begin{theorem}
\label{thm:pac_levy}
Fix $\delta,\beta\in(0,1)$ and suppose that $0\le\rho<\delta$.
Given $R^{(1)},\ldots,R^{(K)}\sim\bbP$ i.i.d.\ and
$R^{(0)}\sim\bbP_{\mathrm{test}}$ satisfying
Assumption~\ref{assump:test_shift_LP}, define
\begin{equation}
\label{eq:pac_lp_estimator}
\bar\alpha_{\mathrm{LP\text{-}CP}}^{\mathrm{test}}(\widehat{\bbP}_K)
:= \Quant_{\,1-\delta+\sqrt{\tfrac{\ln(1/\beta)}{2K}}+\rho}(R^{1:K},\infty)+\eta.
\end{equation}
Assume that the quantile level satisfies
$\frac{1}{K+1}\le 1-\delta+\sqrt{\frac{\ln(1/\beta)}{2K}}+\rho\le\frac{K}{K+1}$.
Then the following calibration-conditional guarantee holds:
\begin{equation}
\label{eq:pac_lp_bound}
\bbP^K\Bigl(
\bbP_{\mathrm{test}}\big(R^{(0)}\le
\bar\alpha_{\mathrm{LP\text{-}CP}}^{\mathrm{test}}(\widehat{\bbP}_K)\big)
\ge 1-\delta
\Bigr)\ge 1-\beta.
\end{equation}
\resultend
\end{theorem}

\begin{proof}
Let
\begin{equation}
p_K := 1-\delta+\sqrt{\frac{\ln(1/\beta)}{2K}}+\rho.
\end{equation}
By Lemma~\ref{lem:2}, applied with miscoverage level $\delta-\rho$,
\begin{equation}
\bbP^K\!\left(
F_{\bbP}\!\left(\Quant_{p_K}(R^{1:K},\infty)\right) \ge 1-\delta+\rho\right) \ge 1-\beta.
\end{equation}
Equivalently,
\begin{equation}
\bbP^K\!\left(
F_{\bbP}\!\left(\Quant_{p_K}(R^{1:K},\infty)\right)-\rho
\ge 1-\delta \right) \ge 1-\beta.
\end{equation}

Under Assumption~\ref{assump:test_shift_LP}, the LP shift relation implies
that, for every $t\in\mathbb{R}$,
\begin{equation}
F_{\bbP}(t-\eta)-\rho
\le
F_{\bbP_{\mathrm{test}}}(t).
\end{equation}
Applying this inequality with
$t=\Quant_{p_K}(R^{1:K},\infty)+\eta$ gives
\begin{equation}
\begin{split}
F_{\bbP}&\!\left(\Quant_{p_K}(R^{1:K},\infty)\right)-\rho \\
&\le
F_{\bbP_{\mathrm{test}}}\!\left(
\Quant_{p_K}(R^{1:K},\infty)+\eta
\right).
\end{split}
\end{equation}
Combining the preceding inequalities proves
\eqref{eq:pac_lp_bound}.
\end{proof}

The preceding result handles the empirical--calibration discrepancy through
CP level inflation. A parallel DRO construction instead accounts for this
discrepancy through the Wasserstein radius $r_K(\beta)$. Under LP
calibration-test shift, the resulting estimator combines an additive
value correction $r_K(\beta)+\eta$ with the quantile-level correction
$\rho$.

\begin{theorem}[\textbf{DRO under LP distribution shift}]
\label{thm:lp_dro_shift}
Fix $\delta,\beta\in(0,1)$ and suppose that $0\le\rho<\delta$.
Further assume that $\mathbb{P}$ is supported on $[a,b] \subset \mathbb{R}$, has
density $f$ satisfying $\inf_{x\in[a,b]}f(x) \ge m>0$, and define $r_K(\beta)$ as in Lemma~\ref{lem:choice-r}.
Given $R^{(1)},\ldots,R^{(K)}\sim\bbP$ i.i.d.\ and
$R^{(0)}\sim\bbP_{\mathrm{test}}$ satisfying
Assumption~\ref{assump:test_shift_LP}, define:
\begin{equation}
\label{eq:pac_dro_estimator}
\bar\alpha_{\mathrm{LP\text{-}DRO}}^{\mathrm{test}}
:=
\Quant_{1-\delta+\rho}(R^{1:K})
+r_K(\beta)+\eta,
\end{equation}
Then the following calibration-conditional guarantee holds:
\begin{equation}
\label{eq:pac_dro_bound}
\bbP^K\Bigl(
\bbP_{\mathrm{test}}\big(R^{(0)}\le
\bar\alpha_{\mathrm{LP\text{-}DRO}}^{\mathrm{test}}\big)
\ge 1-\delta
\Bigr)\ge 1-\beta.
\end{equation}
Moreover, for every realization of the calibration sample,
\begin{equation}
\label{eq:lp_dro_full_ball}
\inf_{\bbQ\in
\mathcal{B}_{r_K(\beta)+\eta,\rho}(\widehat{\bbP}_K)}
\bbQ\left(
R^{(0)}\le
\bar\alpha_{\mathrm{LP\text{-}DRO}}^{\mathrm{test}}
\right)
\ge 1-\delta.
\end{equation}
\resultend
\end{theorem}

\begin{proof}
For any
$\bbQ\in\mathcal{B}_{r_K(\beta)+\eta,\rho}(\widehat{\bbP}_K)$,
the LP shift relation implies that, for every $t\in\mathbb{R}$,
\begin{equation}
F_{\widehat{\bbP}_K}\bigl(t-r_K(\beta)-\eta\bigr)-\rho
\le
F_{\bbQ}(t).
\end{equation}
Applying this inequality with
$t=\Quant_{1-\delta+\rho}(R^{1:K})+r_K(\beta)+\eta$
gives
\begin{equation}
F_{\bbQ}\left(
\bar\alpha_{\mathrm{LP\text{-}DRO}}^{\mathrm{test}}
\right)
\ge
F_{\widehat{\bbP}_K}\left(
\Quant_{1-\delta+\rho}(R^{1:K})
\right)-\rho
\ge 1-\delta.
\end{equation}
Since this holds for every distribution in the ambiguity set,
\eqref{eq:lp_dro_full_ball} follows. Then, by the choice of $r_K(\beta)$,
\begin{equation}
\bbP^K\left(
W_\infty(\bbP,\widehat{\bbP}_K)\le r_K(\beta)
\right)\ge 1-\beta.
\end{equation}
On this event, combining the $W_\infty$ coupling between
$\widehat{\bbP}_K$ and $\bbP$ with the LP coupling between
$\bbP$ and $\bbP_{\mathrm{test}}$ gives
\begin{equation}
\bbP_{\mathrm{test}}
\in
\mathcal{B}_{r_K(\beta)+\eta,\rho}(\widehat{\bbP}_K).
\end{equation}
Therefore, the deterministic guarantee
\eqref{eq:lp_dro_full_ball} applies to $\bbP_{\mathrm{test}}$
with probability at least $1-\beta$, proving
\eqref{eq:pac_dro_bound}.
\end{proof}

In summary, there are two sources of uncertainty in the shifted setting.
First, the finite-sample discrepancy between the calibration data
$\widehat{\bbP}_K$ and the calibration distribution $\bbP$ can be handled
either by CP, through quantile-level inflation, or by DRO, through the
additive Wasserstein radius $r_K(\beta)$. Second, calibration-test
distribution shift can be modeled by placing $\bbP_{\mathrm{test}}$ in a
user-specified ambiguity set around $\bbP$.

Under the $W_\infty$ shift model, calibration-test shift contributes the
same additive value correction $+\eta$ to both the CP and DRO thresholds.
Under the LP shift model with parameters $(\eta,\rho)$, it contributes an
additive value correction $+\eta$ and a quantile-level correction $+\rho$.
Consequently, the LP-shifted CP estimator combines $\rho$ with the usual
CP level inflation, whereas the LP-shifted DRO estimator combines the
quantile-level correction $\rho$ with the additive value correction
$r_K(\beta)+\eta$. Table~\ref{tab:comparison_test_shift} summarizes the resulting estimators and guarantees.

%
% \marginJC{I wonder if it would be possible in Tables I and II to add 1 more column that refers to the specific result, e.g., cf. Theorem~BLA. Maybe using a smaller font in the table to make room?}
%

% \marginJC{It's slightly ugly to use "Quantile" in the table, but $\widehat Q$ in the result. Also, why $\widehat{\bbP}_K \cup \{\infty \}$ in the table but only $\widehat{\bbP}_K$ in the results?}
%

\section{Application Studies in Vision, Language, and Autonomous Driving}
\label{sec:exp_setup_cp_cls}
%
% \marginJC{Ideally, a more informative title than just "Simulations". The previous section is already titled "Numerical Experiments", which is pretty synonym with "Simulations".}
% %

% \NA{Add a brief overview of this section before jumping in.}
% %
% \marginJC{Yes to the overview! And pay attention to explaining what is it that we intend to show/visualize/display}
%

In Section~\ref{sec:examples}, we studied the behavior of the CP and DRO quantile estimators on scalar distributions. We now turn to four prediction settings in which the score distribution is induced by a model and a dataset.

First, we consider image classification (Section~\ref{subsec: image_classification}), which offers an i.i.d. baseline with a large discrete label space. Next, we evaluate the shift-aware estimators of Section~\ref{subsec:distribution_shift} on the ImageNet-C dataset (Section~\ref{subsec:dist_shift}), which adds controlled corruptions to the ImageNet validation images to obtain a test set with distribution shift. Multiple-choice question answering
(Section~\ref{subsec:language_models}) shrinks the label space to four options, so an estimator can more easily include all four, giving a valid but
uninformative prediction set. Finally, we consider a trajectory
prediction problem (Section~\ref{subsec:autonomous_driving}), in which the
discrete label set is replaced by a ball of radius $\bar\alpha$, which can
grow without bound, so no such limit exists. All four settings use $K=1000$ at $\delta=\beta=0.1$, for which Remark~\ref{rem:cp_finite_sample} gives $K_{\min}=135$ for Lemmas~\ref{lem:2} and~\ref{lem:4}, so every considered setting has sufficient calibration samples to exceed the vacuity threshold.

In each study, we report mean coverage, which tests the marginal guarantee, the two-level rate $\bbP^K\big(\bbP(R^{(0)}\le\bar\alpha)\ge 1-\delta\big)$, which tests the calibration-conditional guarantee \eqref{eq:cal_cond_goal}, and the prediction-set
size needed to attain the guarantees. We compare split CP in \eqref{eq:marginal_coverage_goal}, which targets marginal coverage only, against the calibration-conditional corrections of Lemmas~\ref{lem:2}, \ref{lem:3} and~\ref{lem:4} and against DRO (Lemma~\ref{lem:dro_quantile_refined}).

\subsection{Image Classification}
\label{subsec: image_classification}

% \marginJC{What happens in any of the applications in these subsections for low $K$, i.e., when CP becomes vacuous? Should we illustrate/discuss that too?}

\textbf{Problem setup.}
We consider $N$-class image classification with input space 
$\mathcal{X}$ and label space $\mathcal{Y}=\{1,\dots,N\}$.
A pre-trained ResNet-152 classifier $f$ outputs 
class probabilities $f(x)=(f_1(x),\dots,f_N(x))$, where 
$f_y(x)\approx \mathbb{P}(Y=y\mid X=x)$.
The dataset is split into a \emph{calibration set} 
$\mathcal{D}_{\text{cal}}=\{(X_i,Y_i)\}_{i=1}^K$ and an \emph{evaluation set} 
$\mathcal{D}_{\text{eval}}=\{(X_j^{\text{eval}},Y_j^{\text{eval}})\}_{j=1}^{n_{\text{eval}}}$, 
both obtained by a uniformly random split of the dataset, hence exchangeable.

For each $(X_i,Y_i)\in\mathcal{D}_{\text{cal}}$, define the nonconformity score
\[
R^{(i)} \;=\; 1 - f_{Y_i}(X_i),
\]
which measures how nonconforming the true label is with respect to the 
model's prediction. A small $R^{(i)}$ indicates high confidence in the 
correct class; a large value corresponds to greater uncertainty or 
misclassification.

For a new input $x$, our goal is to construct a prediction set 
$C(x)\subseteq\{1,\dots,N\}$ that contains the true label with probability 
at least $1-\delta$:
\[
\mathbb{P}\big(Y_{\text{test}} \in C(X_{\text{test}})\big) \;\ge\; 1-\delta,
\]
under the assumption that calibration and test samples are exchangeable. Coverage alone is trivially achieved by the full label set $C(x)=\{1,\dots,N\}$; the objective is therefore to attain this
coverage while keeping the prediction set \emph{as small as possible}. The prediction set
%
% \marginJC{Why conformal here? Shouldn't it simply be "the prediction set"? Later, we will determine if we do it with CP, or w/ DRO, or otherwise, no?}
%
is defined as
\[
C(x) = \big\{\, y \in \{1,\dots,N\} : f_y(x) \ge 1 - \bar{\alpha} \,\big\},
\]
where $\bar{\alpha}$ is a calibrated threshold. Intuitively, $C(x)$ 
collects all labels whose predicted probabilities exceed $1-\bar{\alpha}$, 
thereby guaranteeing the finite-sample coverage property 
$\mathbb{P}(Y_{\text{test}} \in C(X_{\text{test}})) \ge 1-\delta$, 
regardless of the underlying data distribution or classifier 
calibration~\cite{angelopoulos2021gentle}.

In the standard \emph{split conformal prediction} (SCP) framework \cite{shafer2008tutorial, angelopoulos2021gentle}, this 
threshold is the empirical quantile
\[
\bar{\alpha}_{\mathrm{SCP}} \;=\;
\Quant_{\,1-\delta}\!\big(R^{1:K},\infty\big),
\]
which calibrates the model's confidence scores to achieve the desired marginal coverage \eqref{eq:marginal_coverage_goal}. In contrast, we can also define $\bar{\alpha}$ 
according to either the \emph{calibration-conditional} conformal prediction 
in~\eqref{eq:cp_cond} and~\eqref{eq: cali_cp_lemma3}, or the 
\emph{distributionally robust} variant in~\eqref{eq: alpha_dro_bound}, 
which strengthen the coverage guarantees under calibration uncertainty.

\textbf{Experimental setup.}
We evaluate the CP and DRO estimators on the ImageNet validation set using the pre-trained ResNet-152 classifier. The dataset contains $50000$ samples. 
We randomly select $K=1000$ samples for calibration and use the remaining 
$n_{\text{eval}}=49000$ for evaluation. The target miscoverage level is $\delta=0.1$ (i.e., $0.9$ nominal coverage). For the DRO variants we report two Wasserstein radii, $r_K\in\{0.02,0.03\}$, spanning the regime where DRO transitions from compact-but-unreliable to conservative-but-bloated prediction sets.

\textbf{Results.}
We conduct $1000$ independent trials with different random splits of 
calibration and evaluation data. Figure~\ref{fig:imagenet_coverage} reports 
the mean empirical coverage across trials. As expected, all six methods
achieve marginal coverage at or above the nominal $0.9$. The DRO radii shown are design choices rather than the certified $r_K(\beta)$, so this is an empirical observation and not a verification of the guarantee. 

\begin{figure}[htb]
    \centering
    \includegraphics[width=\linewidth]{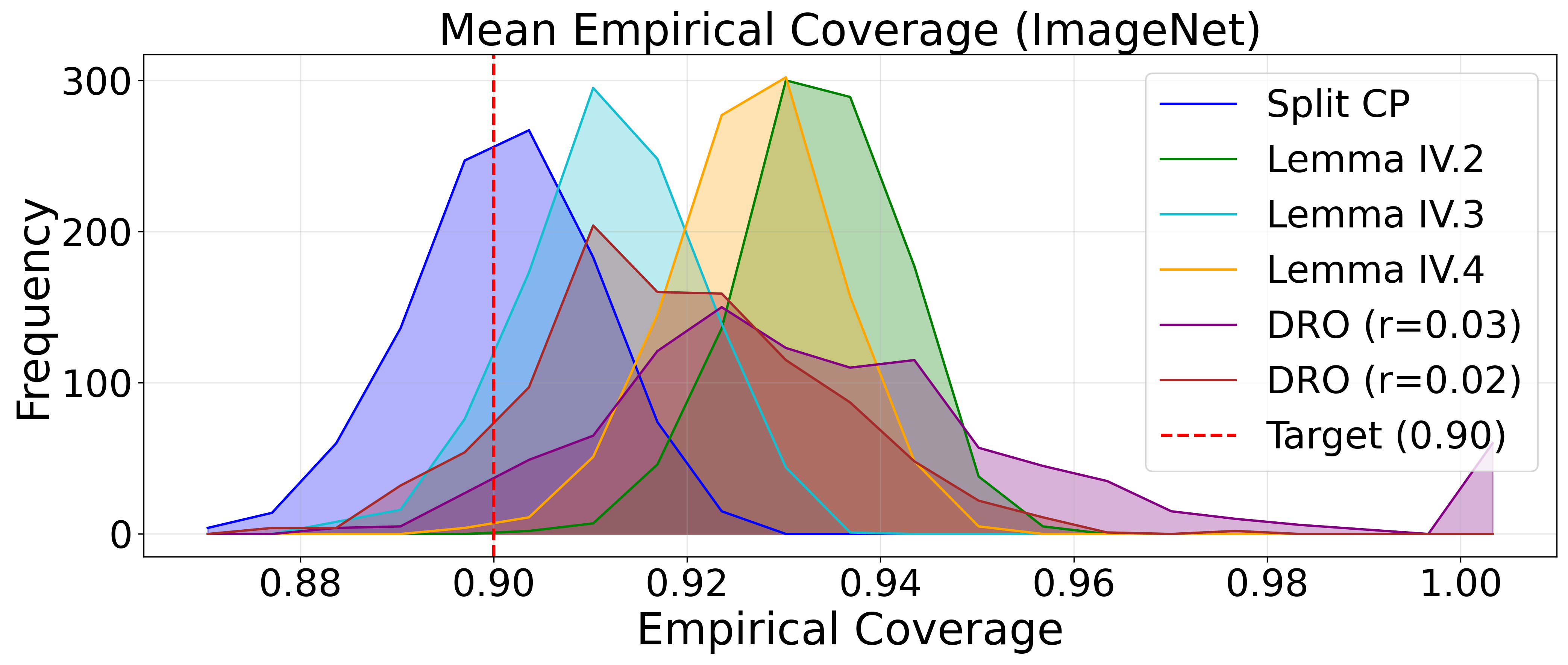}
    \caption{Mean empirical coverage across $1000$ random trials on ImageNet with $\delta = 0.1$. All methods achieve marginal coverage at or above the nominal $0.9$. Split CP is centered on the target, while the
    calibration-conditional and DRO variants concentrate above it.}
    \label{fig:imagenet_coverage}
    \vspace{-1.5ex}
\end{figure}

\begin{figure}[htb]
    \centering
    \includegraphics[width=\linewidth]{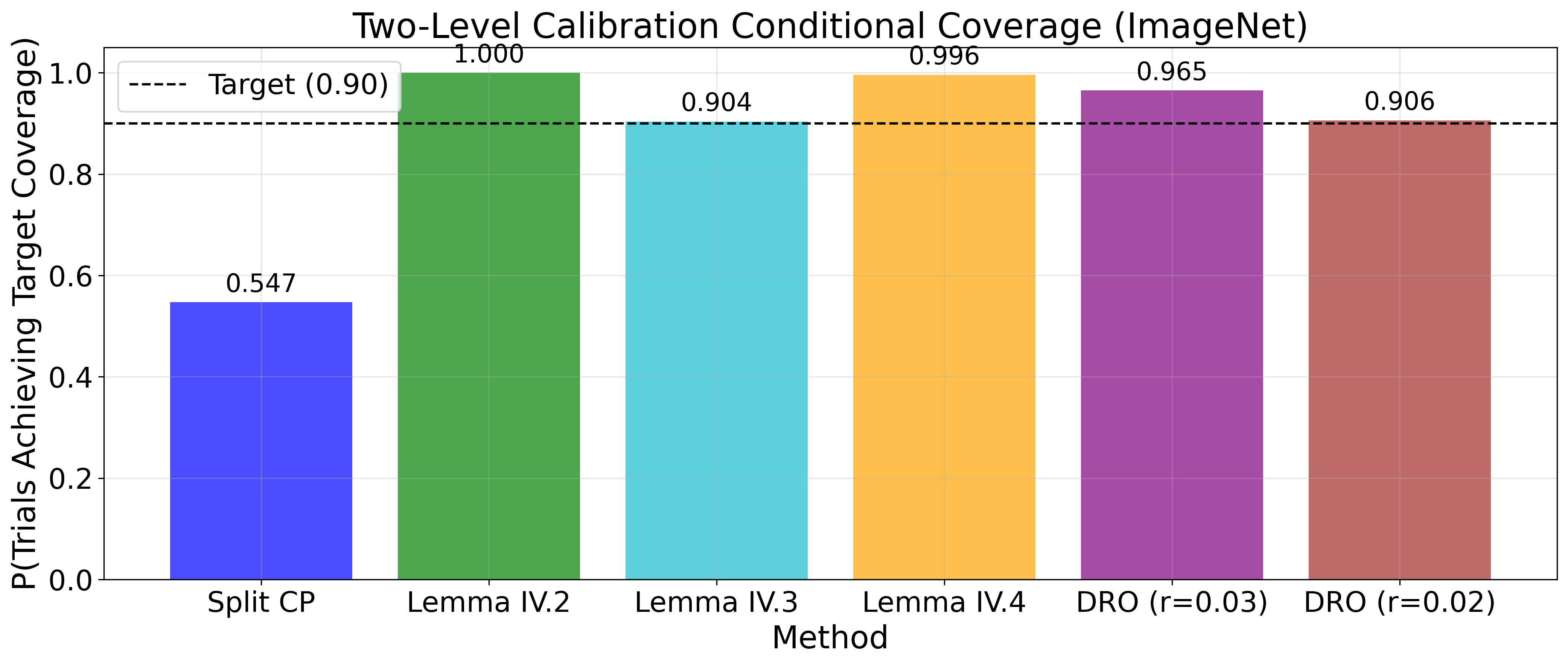}
    \caption{Two-level calibration-conditional coverage at
    $(\delta,\beta)=(0.1,0.1)$. Each of the $1000$ trials draws one calibration
    set and so yields one conditional coverage
    $\mathbb{P}(R^{(0)} \le \bar{\alpha})$; plotted is the fraction of trials for
    which this is at least $1-\delta$, which estimates $\mathbb{P}^K$ in \eqref{eq:cal_cond_goal} and must reach $1-\beta$ for the guarantee to hold.
    Split CP reaches only $0.547$, while Lemma~\ref{lem:2} and Lemma~\ref{lem:4}
    achieve $1.000$ and $0.996$, Lemma~\ref{lem:3} achieves $0.904$, and DRO requires a sufficiently large radius.}
    \label{fig:imagenet_two_level}
    \vspace{-2ex}
\end{figure}

To evaluate the stronger conditional guarantee, for each trial we test the 
two-level criterion
\[
\mathbb{P}^K\!\big(\mathbb{P}(R^{(0)} \le \bar{\alpha}(R^{1:K})) \ge 1-\delta\big) \ge 1-\beta,
\]
with $(\delta,\beta)=(0.1,0.1)$. Figure~\ref{fig:imagenet_two_level} 
highlights a key limitation of standard SCP: although it achieves the
target marginal coverage on average ($0.901$), it satisfies the two-level
guarantee with rate only $0.547$. This shortfall arises because SCP
does not account for randomness in the calibration scores, limiting its
robustness under the stronger conditional criterion. In contrast, both
calibration-conditional CP methods (Lemmas~\ref{lem:2} and~\ref{lem:4}) consistently
satisfy the two-level guarantee in essentially all trials ($1.000$ and
$0.996$), while the DRO variants require a sufficiently large radius
($r_K = 0.03$ achieves $0.965$, $r_K=0.02$ only $0.906$). Lemma~\ref{lem:3}
behaves differently from the other two corrections: its condition is exact
rather than a relaxation, so its two-level rate stays near the target $1-\beta$ ($0.904$) instead of saturating at $1$, and its sets are correspondingly smaller. Because
coverage is itself estimated on a finite evaluation set, the measured rate of
such a tight method sits at the nominal value and can fall marginally on either
side of it.

\begin{table}[htb]
    \centering
    \small
    \caption{Coverage, two-level satisfaction probability, and mean 
    prediction-set size across $1000$ trials on ImageNet 
    ($K=1000$, $\delta=0.1$, $\beta=0.1$).
    $^{*}$DRO ($r_K=0.03$) yields a vacuous prediction set (size $1000$)
    in a fraction $0.060$ of trials where the additive radius pushes the threshold
    $\bar{\alpha}$ above $1$; the reported mean and std exclude these 
    trials. Including all trials, the mean rises to $63.1$ with 
    std $236.7$.}
    \label{tab:imagenet_set_size}    
    \setlength{\tabcolsep}{4pt}
    \resizebox{\linewidth}{!}{\begin{tabular}{lccc}
    \toprule
    Method & Mean coverage & $\mathbb{P}^K(\text{cov}\geq 1{-}\delta)$ $\uparrow$ & Set size $\downarrow$ \\
    \midrule
    Split CP        & $0.901$ & $0.547$ & $1.8 \pm 0.2$ \\
    Lemma~\ref{lem:2}      & $0.934$ & $1.000$ & $2.7 \pm 0.3$ \\
    Lemma~\ref{lem:3}      & $0.912$ & $0.904$ & $2.0 \pm 0.2$ \\
    Lemma~\ref{lem:4}      & $0.927$ & $0.996$ & $2.4 \pm 0.3$ \\
    DRO ($r_K=0.03$)  & $0.935$ & $0.965$ & $3.0 \pm 1.9^{*}$ \\
    DRO ($r_K=0.02$)  & $0.919$ & $0.906$ & $2.3 \pm 0.5$ \\
    \bottomrule
    \end{tabular}}
\end{table}

Table~\ref{tab:imagenet_set_size} shows the conservativeness/robustness trade-off across the methods. The calibration-conditional CP methods (Lemmas~\ref{lem:2} and~\ref{lem:4}) emerge as the sweet spot because they satisfy the stronger two-level guarantee in essentially all trials while keeping prediction sets small (mean size $2.4$ to $2.7$ out of $1000$ classes), while
Lemma~\ref{lem:3} gives up that margin for sets of mean size $2.0$. Split CP
gives the smallest sets (mean $1.8$) but fails the two-level guarantee
at rate $0.547$. The DRO variant $r = 0.02$ is
similarly compact and clears the criterion only marginally ($0.906$).
DRO with $r = 0.03$ satisfies the guarantee with rate $0.965$ and a
comparable mean set size on non-degenerate trials ($3.0$), but in a fraction
$0.060$ of trials the calibrated threshold $\bar{\alpha}$ exceeds $1$ and the
prediction set degenerates to all $1000$ classes.

This degenerate behavior reflects a structural limitation of the 
additive form $\bar{\alpha}_{\text{DRO}} = \bar{\alpha}_{\text{emp}} + r$:
the nonconformity scores are bounded in $[0,1]$, so if the empirical
%
% \marginJC{Aren't we missing here "if" so that this reads "so IF the empirical..."? }
%
quantile $\bar{\alpha}_{\text{emp}}$ is already close to the upper end  of its range under moderate calibration noise then an additive offset $r$ can push the threshold past $1$, at which point any class is admitted into the prediction set. Clipping $\bar{\alpha}$ at $1$ does not help because the prediction set $C(x)=\{y: f_y(x)\ge 1-\bar{\alpha}\}$ then admits every label, so it remains vacuous. A remedy would instead need a data-adaptive radius that scales with the
local score density near the $(1-\delta)$-quantile, with its own validity argument. 

Overall, this experiment shows that calibration-conditional CP is the most efficient route to satisfying the two-level guarantee under i.i.d. data.

\subsection{Image Classification Under Distribution Shift}
\label{subsec:dist_shift}
%
% \marginJC{Title here should be something like "Image Classification under Distribution Shift"?}
%

\textbf{Problem setup.}
We next test robustness under distribution shift using ImageNet-C
\cite{hendrycks2019robustness}, which applies controlled corruptions to the
ImageNet validation images. We use three of its corruption families spanning distinct
mechanisms: Gaussian noise, motion blur, and fog, each at the five severity levels defined by ImageNet-C~\cite{hendrycks2019robustness}, illustrated in Figure~\ref{fig:corruption_examples}. 
\begin{figure}[htb]
    \centering
    \includegraphics[width=\linewidth]{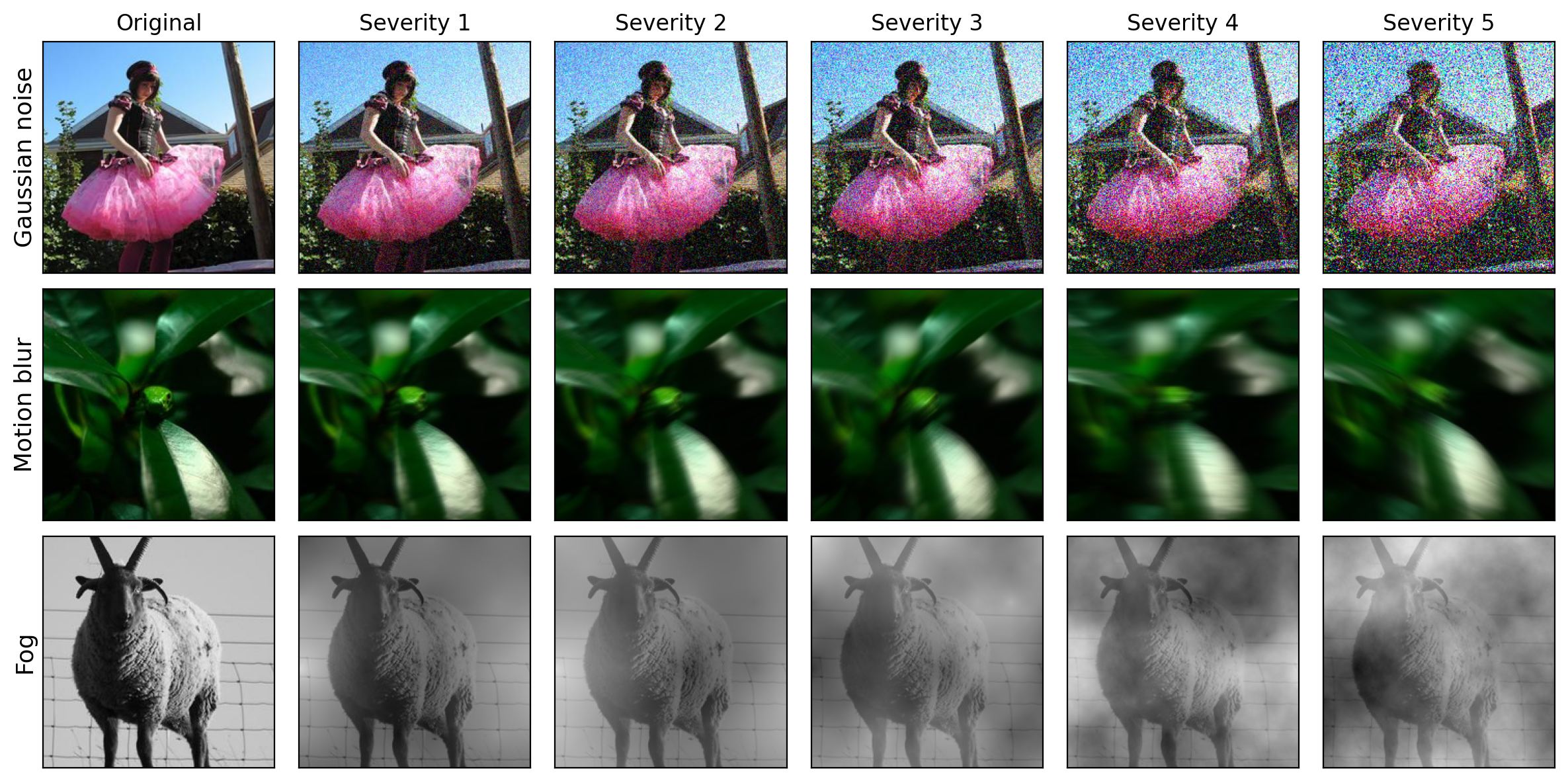}
    \caption{ImageNet-C corruptions used in our experiments. Each row applies one corruption family to a clean validation image (leftmost) at increasing
    severity ($1$ to $5$). Severity controls the magnitude of the shift between
    the clean calibration distribution and the corrupted evaluation distribution.}
    \label{fig:corruption_examples}
\end{figure}

\begin{figure*}[htb]
    \centering
    \includegraphics[width=\linewidth]{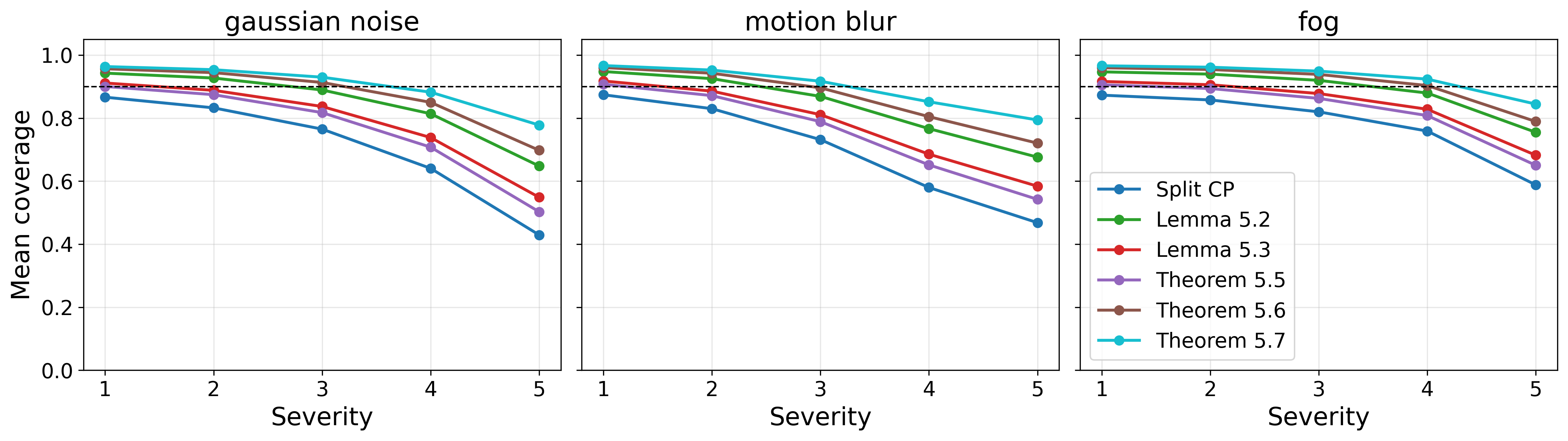}
    \caption{Mean coverage on ImageNet-C versus severity ($\eta=0.002$,
    $\rho=0.020$), for Gaussian noise, motion blur, and fog. Split CP degrades
    monotonically below the $0.90$ target; the shift-aware methods hold higher
    coverage, highest for Theorem~\ref{thm:lp_dro_shift} and then
    Theorem~\ref{thm:pac_levy}. The margin of
    Theorem~\ref{thm:lp_dro_shift} is partly obtained with vacuous sets
    (Table~\ref{tab:imagenet_c_shift}).}
    \label{fig:imagenet_c_graded}
\end{figure*}

\begin{figure*}[htb]
    \centering
    \includegraphics[width=\linewidth]{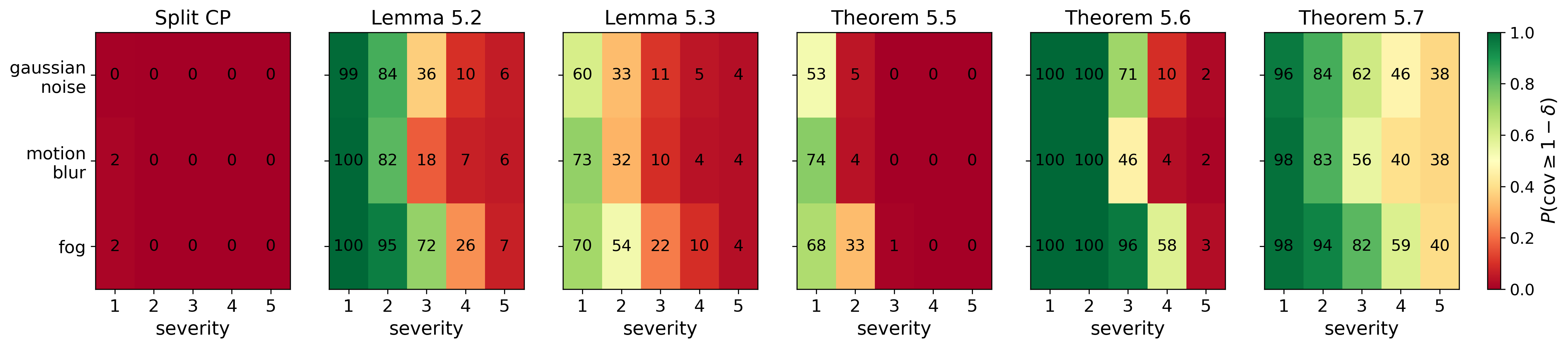}
    \caption{Two-level guarantee probability
    $\mathbb{P}^K(\mathbb{P}_{\mathrm{test}}\{R^{(0)}\le\bar\alpha\}\ge 1-\delta)$
    over the $3\times5$ (corruption, severity) grid ($\eta=0.002$, $\rho=0.020$).
    Green indicates the guarantee is met in nearly all trials. Split CP fails everywhere under shift. Theorem~\ref{thm:pac_levy} holds the guarantee at low severity and degrades gracefully. The marginal method
    (Theorem~\ref{thm:lp_robust_cp}) does not satisfy the two-level criterion, as expected of a marginal guarantee. Theorem~\ref{thm:lp_dro_shift} degrades the most slowly of all, retaining roughly $0.4$ at the highest severity, but does so by returning vacuous sets.}
    \label{fig:imagenet_c_heatmap}
\end{figure*}

%
% \marginJC{How about also checking the performance of our shift-unaware estimators? I bet DRO would still well -- depending on the conservativeness of the radius.}
%
\textbf{Experimental setup.}
In each trial, we draw $K=1000$ clean images for calibration and use the corrupted versions of the remaining $n_{\text{eval}}=49000$ for evaluation. We keep $\delta=0.1$, $\beta=0.1$, and report $200$ independent trials.
We compare the five shift-aware estimators of
Table~\ref{tab:comparison_test_shift}: CP under $W_\infty$ shift
(Lemma~\ref{lem:cp_shift}), DRO under $W_\infty$ shift
(Lemma~\ref{lem:dro_shift}), the marginal LP-robust CP of
\cite{aolaritei2025conformal} (Theorem~\ref{thm:lp_robust_cp}), and our
calibration-conditional extensions under LP shift
(Theorem~\ref{thm:pac_levy} and its DRO counterpart
Theorem~\ref{thm:lp_dro_shift}), together with split CP as a no-adjustment baseline.

The two shift models take different budgets. For the $W_\infty$ methods (Lemmas~\ref{lem:cp_shift} and~\ref{lem:dro_shift}) we set $\eta=0.008$. For the LP methods (Theorems~\ref{thm:lp_robust_cp}, \ref{thm:pac_levy} and~\ref{thm:lp_dro_shift}) we set $\eta=0.002$ and $\rho=0.020$, chosen so that Theorem~\ref{thm:lp_robust_cp} attains $1-\delta$ marginal coverage at severity~1. The $W_\infty$ budget is larger because it has no $\rho$ to absorb the level part of the shift.

\textbf{Results.}
Figure~\ref{fig:imagenet_c_graded} reports mean coverage versus severity, and
Figure~\ref{fig:imagenet_c_heatmap} the two-level guarantee probability for all three corruptions and five severities. Split CP, which has no shift mechanism, degrades sharply:
its mean coverage falls from $0.87$ at severity~1 to $0.50$ at severity~5
(Gaussian noise), and it satisfies the two-level guarantee in essentially no
shifted cell. The shift-aware methods hold higher coverage. Among them, Theorem~\ref{thm:pac_levy} is the most robust, maintaining the two-level guarantee at severity~1--2 across all three corruptions.

Table~\ref{tab:imagenet_c_shift} details a representative mild cell (motion blur, severity~1). Two points stand out. First, the marginal method
(Theorem~\ref{thm:lp_robust_cp}) achieves its target marginal coverage
($0.907$) with small sets (median $2.5$), but satisfies the stronger
two-level criterion with rate only $0.745$, which illustrates the gap between marginal robustness and the calibration-conditional guarantee. Second, Theorem~\ref{thm:pac_levy} closes this gap by satisfying the two-level guarantee in all trials with usable prediction sets (median
$6.8$ of $1000$ classes, $0.000$ vacuous). Its DRO counterpart (Theorem~\ref{thm:lp_dro_shift}) shows the limit of the value-space route on a bounded score: it applies the level correction $\rho$ and the additive radius
together, which drives $\bar{\alpha}$ past $1$ in $34\%$ of trials even at the reduced radius $r_K=0.02$, so its median set of $10.7$ classes is obtained at the cost of frequent degeneracy. On the continuous nuScenes score, the same estimator is instead the tightest of the shift-aware methods (Table~\ref{tab:nuscenes_shift}).

\begin{figure*}[htb]
    \centering
    \includegraphics[width=\linewidth]{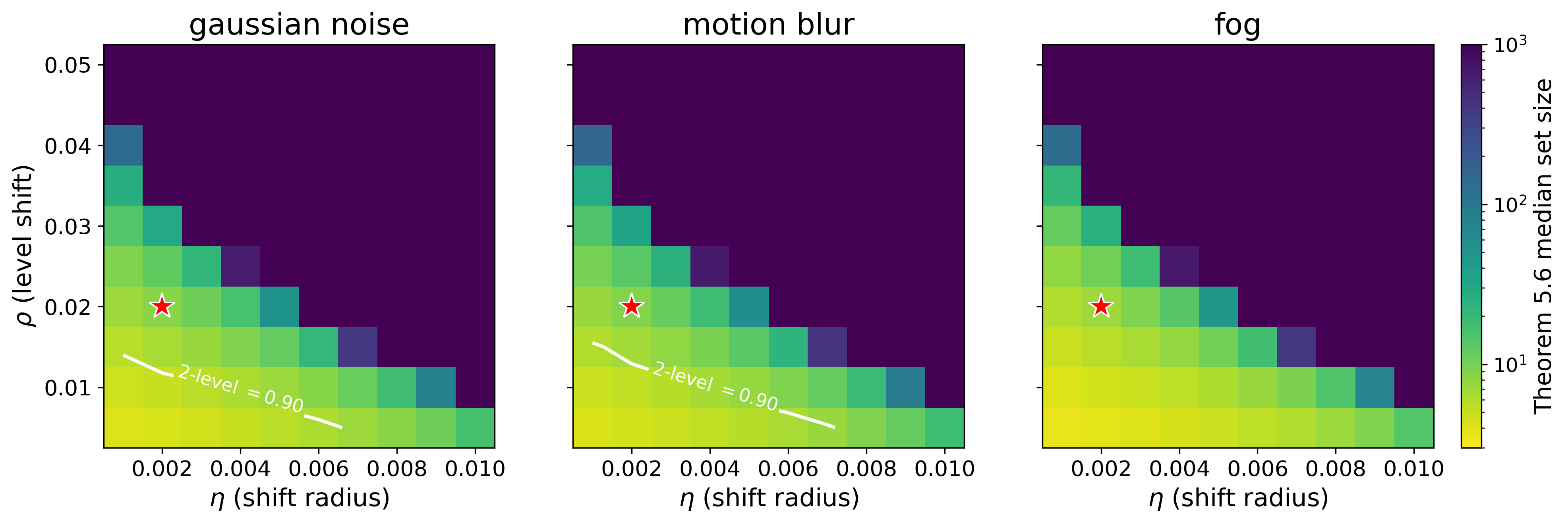}
    \caption{Effect of the LP ambiguity-set parameters $(\eta,\rho)$ on
    Theorem~\ref{thm:pac_levy} at severity~2, for each corruption. Color is the
    median prediction-set size (log scale); the white curve marks where the measured two-level rate equals $1-\beta$; the red star marks $(\eta,\rho)=(0.002,0.020)$, just above the curve in the low-set-size region. The measured rate exceeds $1-\beta$ above this curve but the set size grows steeply with $\eta$ (the additive radius saturates
    the threshold toward the all-class set) and more gently with $\rho$ at small
    $\eta$. For fog, the guarantee is already met at the smallest budget shown, so its curve lies below the plotted range and no contour appears.}
    \label{fig:imagenet_c_rrho}
\end{figure*}

Figure~\ref{fig:imagenet_c_rrho} examines how the LP parameters $(\eta,\rho)$ affect Theorem~\ref{thm:pac_levy} across the three corruptions at severity~2. The measured two-level rate exceeds $1-\beta$ above the white curve but the median prediction-set size grows steeply with the Wasserstein radius $\eta$. This occurs because the additive value correction pushes the threshold toward the all-class set. In contrast, the set size increases relatively gently with the level parameter $\rho$ when $\eta$ is small. Robustness is therefore most efficiently supplied through $\rho$ while keeping $\eta$ small. This contrasts with the pure $W_\infty$ methods, whose only lever is $\eta$, and explains their higher set sizes and vacuity in Table~\ref{tab:imagenet_c_shift}.

\begin{table}[t]
    \centering
    \small
    \caption{Results at motion blur severity~1 ($K=1000$, $\delta=0.1$, $\beta=0.1$, $200$ trials). The two shift models carry their own budgets, recorded in the $(\eta,\rho)$ column. ``Mean coverage'' is the average coverage over trials; $\mathbb{P}^K(\text{cov}\geq 1{-}\delta)$ is the
    calibration-conditional success rate
    $\mathbb{P}^K(\mathbb{P}_{\mathrm{test}}(R^{(0)}\le\bar\alpha)\ge1-\delta)$;
    ``Set size'' is the median prediction-set size; ``Vac.''\ is the fraction
    of trials whose threshold $\bar\alpha$ exceeds $1$. The
    marginal method (Theorem~\ref{thm:lp_robust_cp}) attains marginal coverage ($0.907$) but not the two-level guarantee ($0.745$). Theorem~\ref{thm:pac_levy} attains both with usable sets. Its DRO
    counterpart (Theorem~\ref{thm:lp_dro_shift}) meets the guarantee only
    degenerately on this bounded score: stacking the level correction $\rho$ on
    the additive radius still pushes $\bar\alpha$ above $1$ in $34\%$ of trials.
    The DRO radius is set to $r_K=0.02$, carried over from
    Table~\ref{tab:imagenet_set_size} as the smallest i.i.d.\ radius meeting the
    two-level guarantee (it is a design choice in the sense of
    Section~\ref{sec:choosing_radius}, not the certified value).}
    \label{tab:imagenet_c_shift}    
    \setlength{\tabcolsep}{4pt}
    \resizebox{\linewidth}{!}{\begin{tabular}{llcccc}
    \toprule
    Method & $(\eta,\rho)$ & Mean coverage & $\mathbb{P}^K(\text{cov}\geq 1{-}\delta)$ $\uparrow$ & Set size $\downarrow$ & Vac.\ $\downarrow$ \\
    \midrule
    Split CP                       & -- & $0.874$ & $0.015$ & $1.8$ & $0.000$ \\
    \midrule
    \multicolumn{6}{l}{\emph{$W_\infty$ shift model} (Assumption~\ref{assump: test_shift_wasserstein})} \\
    Lemma~\ref{lem:cp_shift}       & $(0.008,-)$ & $0.948$ & $1.000$ & $4.4$ & $0.030$ \\
    Lemma~\ref{lem:dro_shift}      & $(0.008,-)$ & $0.918$ & $0.730$ & $2.7$ & $0.035$ \\
    \midrule
    \multicolumn{6}{l}{\emph{L\'evy--Prokhorov shift model} (Assumption~\ref{assump:test_shift_LP})} \\
    Theorem~\ref{thm:lp_robust_cp} & $(0.002,0.02)$ & $0.907$ & $0.745$ & $2.5$ & $0.000$ \\
    Theorem~\ref{thm:pac_levy}     & $(0.002,0.02)$ & $0.960$ & $1.000$ & $6.8$ & $0.000$ \\
    Theorem~\ref{thm:lp_dro_shift} & $(0.002,0.02)$ & $0.967$ & $0.980$ & $10.7$ & $0.340$ \\
    \bottomrule
    \end{tabular}}
    \vspace{-2ex}
\end{table}

Distribution shift sharpens the i.i.d. findings of Section~\ref{subsec: image_classification}. Split CP loses coverage outright. The marginal LP-robust method achieves marginal coverage but not the calibration-conditional guarantee, while Theorem~\ref{thm:pac_levy} delivers the two-level guarantee with usable sets at mild-to-moderate corruption. As severity in the distribution shift grows, all methods degrade. A fixed $(\eta,\rho)$ budget eventually cannot absorb the shift, and a larger budget raises the threshold and so enlarges the prediction sets. Matching the budget to the anticipated shift magnitude, thus, trades efficiency for robustness, with Theorem~\ref{thm:pac_levy} being at the most favorable point of this trade-off.

\subsection{Language Model Question Answering}
\label{subsec:language_models}

\textbf{Problem setup.}
We next apply the uncertainty quantification methods to a multiple-choice question answering task using a language model on the MMLU dataset \cite{hendryckstest2021}. We evaluate on the standard test split of $14042$ questions across $57$ academic subjects, each with four answer options $\mathcal{Y}=\{A,B,C,D\}$. A pre-trained Qwen2.5-7B-Instruct model~\cite{yang2024qwen25} reads each question and
its four options and produces a probability $f_y(x)$ over the options from the
next-token logits of the choice letters. This model attains $71.7\%$ top-1 accuracy. As before, we define the nonconformity score as
\[
R^{(i)} = 1 - f_{Y_i}(X_i),
\]
and the prediction set as $C(x)=\{y\in\mathcal{Y}: f_y(x)\ge 1-\bar\alpha\}$.

\textbf{Experimental setup.}
We pool all $14042$ questions and, in each trial, draw $K=1000$ for
calibration and use the remaining $n_{\text{eval}}=13042$ for evaluation, so the calibration and test samples are exchangeable. We set
$\delta=0.1$ and $\beta=0.1$, and report $1000$ independent trials. We compare
Split CP, the three calibration-conditional CP variants Lemmas~\ref{lem:2},
\ref{lem:3} and~\ref{lem:4}), and the $W_\infty$-DRO threshold at radii $r_K\in\{0.002,0.003\}$. Because the four-option scores concentrate near
$1$, the DRO radii here are an order of magnitude smaller than in the ImageNet
study. 

\textbf{Results.}
Figure~\ref{fig:mmlu_coverage} reports the mean empirical coverage across trials and
Figure~\ref{fig:mmlu_two_level}  the two-level guarantee probability. As in the image setting, every method attains marginal coverage at or above the nominal level of $0.9$. The two-level guarantee, however, separates the methods sharply (Table~\ref{tab:mmlu}). Split CP, which controls marginal but not calibration-conditional coverage, attains the target marginal coverage ($0.901$) but satisfies the two-level criterion with rate only $0.557$. The calibration-conditional methods (Lemma~\ref{lem:2}, Lemma~\ref{lem:4}) satisfy it in essentially all trials while keeping the prediction sets compact (mean size $\approx 2.0$ of $4$ options), and Lemma~\ref{lem:3} again tracks the nominal level ($0.891$) with the smallest
calibration-conditional set ($1.84$). The DRO
threshold interpolates between these regimes as $r_K$ grows: $r_K=0.002$ reaches
$0.887$ and $r_K=0.003$ reaches $0.955$. This
improvement, however, comes from a constant additive radius that inflates
every prediction set uniformly. Already at $r_K=0.003$, DRO returns the full four-option set in a fraction $0.160$ of trials (Table~\ref{tab:mmlu}), whereas the
calibration-conditional methods are never vacuous. The same additive-radius
limitation appears in the image experiments
(cf. Sections~\ref{subsec: image_classification} and~\ref{subsec:dist_shift}): a
constant $r_K$ cannot adapt to the score distribution, and calibration-conditional
CP attains the two-level guarantee with smaller sets. 

\begin{figure}[htb]
    \centering
    \includegraphics[width=\linewidth]{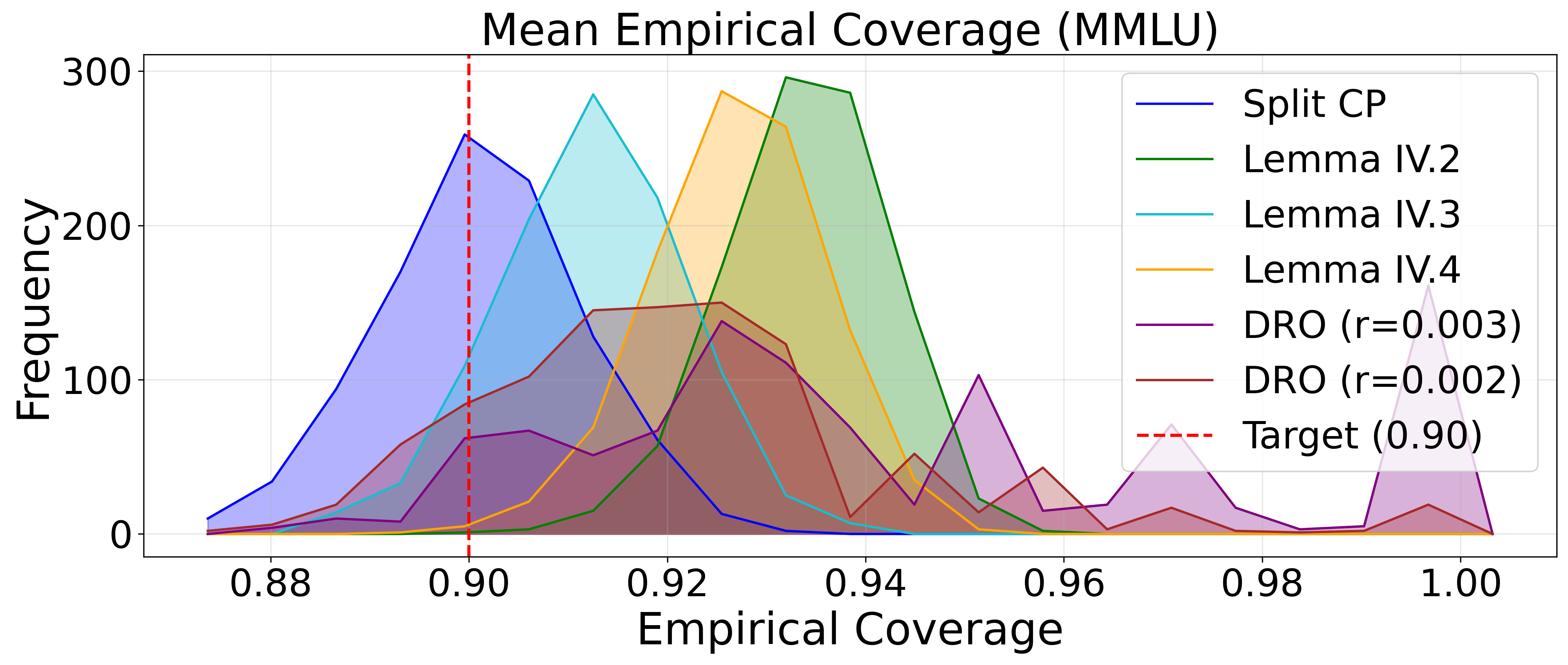}
    \caption{Mean empirical coverage across $1000$ random calibration/evaluation
    splits of MMLU with $\delta=0.1$. All methods achieve marginal coverage at or above the nominal $0.9$. Split CP concentrates on the target while the calibration-conditional and DRO variants concentrate above it.}
    \label{fig:mmlu_coverage}
\end{figure}

\begin{figure}[htb]
    \centering
    \includegraphics[width=\linewidth]{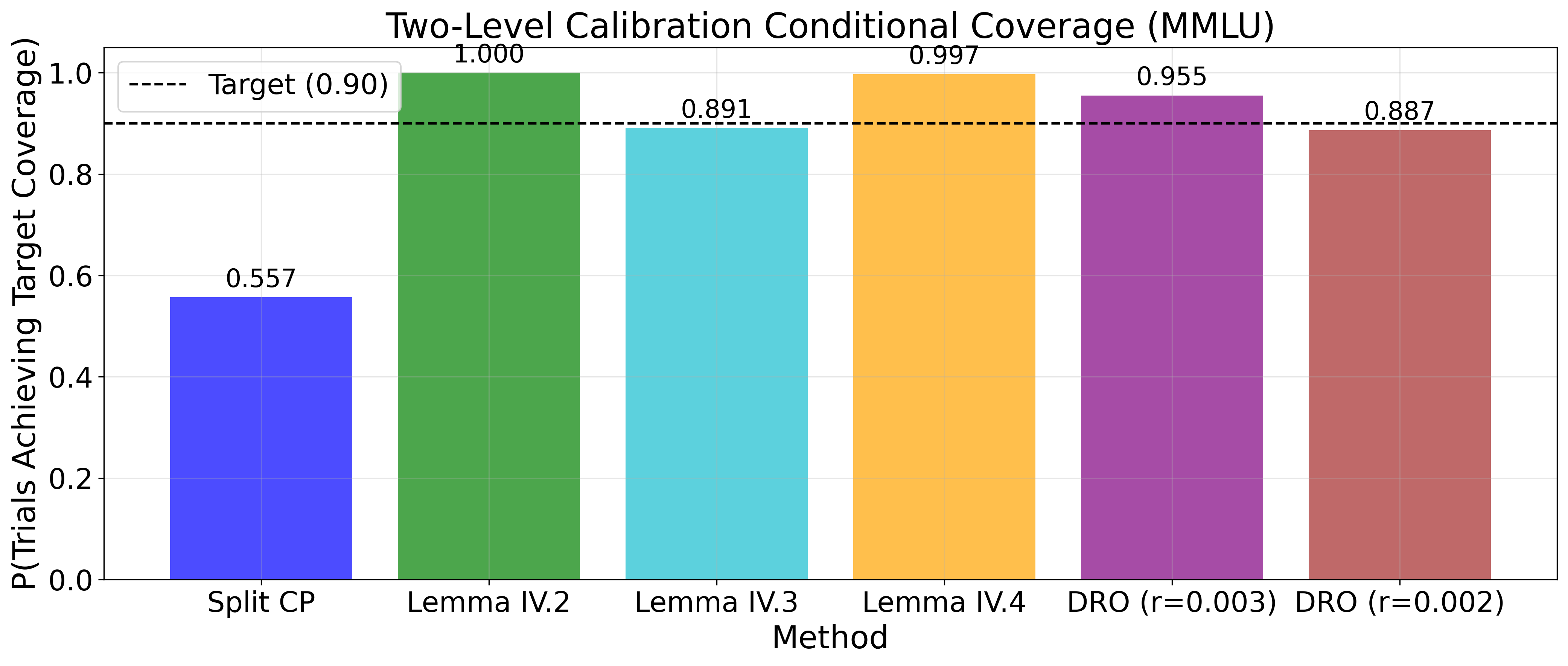}
    \caption{Two-level calibration-conditional coverage on MMLU at
    $(\delta,\beta)=(0.1,0.1)$: the fraction of the $1000$ trials whose
    calibration set attains conditional coverage
    $\mathbb{P}(R^{(0)}\le\bar\alpha)\ge1-\delta$, which estimates $\mathbb{P}^K$ in \eqref{eq:cal_cond_goal} and must reach $1-\beta$ for the guarantee to hold. Split CP reaches only $0.557$,
    Lemma~\ref{lem:2} and Lemma~\ref{lem:4} achieve $1.000$, Lemma~\ref{lem:3} tracks the
    nominal level, and DRO requires a sufficiently large radius.}
    \label{fig:mmlu_two_level}
\end{figure}

\begin{table}[htb]
    \centering
    \small
    \caption{MMLU multiple-choice QA with Qwen2.5-7B-Instruct ($K=1000$,
    $\delta=0.1$, $\beta=0.1$, $1000$ trials, i.i.d.). ``Set size''
    is the mean $\pm$ std prediction-set size, out of four options. Split CP
    attains marginal but not two-level coverage; the calibration-conditional
    methods (Lemma~\ref{lem:2}, Lemma~\ref{lem:4}) attain the two-level guarantee with compact
    sets and Lemma~\ref{lem:3} tracks the nominal level with smaller sets still. $^{*}$The constant additive DRO radius returns the full four-option (vacuous) set in fractions $0.019$ and $0.160$ of trials at $r_K=0.002$ and $r_K=0.003$ respectively; as in Table~\ref{tab:imagenet_set_size}, the reported mean and std exclude these trials. Including all trials they are $1.97 \pm 0.34$ and $2.38 \pm 0.75$.}
    \label{tab:mmlu}    
    \setlength{\tabcolsep}{5pt}
    \resizebox{\linewidth}{!}{\begin{tabular}{lccc}
    \toprule
    Method & Mean coverage & $\mathbb{P}^K(\text{cov}\geq 1{-}\delta)$ $\uparrow$ & Set size $\downarrow$ \\
    \midrule
    Split CP        & $0.901$ & $0.557$ & $1.76 \pm 0.07$ \\
    Lemma~\ref{lem:2}      & $0.934$ & $1.000$ & $2.04 \pm 0.08$ \\
    Lemma~\ref{lem:3}      & $0.912$ & $0.891$ & $1.84 \pm 0.07$ \\
    Lemma~\ref{lem:4}      & $0.927$ & $0.997$ & $1.97 \pm 0.08$ \\
    DRO ($r_K=0.003$) & $0.944$ & $0.955$ & $2.07 \pm 0.27^{*}$ \\
    DRO ($r_K=0.002$) & $0.922$ & $0.887$ & $1.93 \pm 0.19^{*}$ \\
    \bottomrule
    \end{tabular}}
\end{table}

\subsection{Autonomous Driving Trajectory Prediction}
\label{subsec:autonomous_driving}

\textbf{Problem setup.}
Our final experiment moves from discrete classification to a
\emph{continuous-output} task: vehicle trajectory prediction on
nuScenes~\cite{caesar2020nuscenes}. For each agent, we observe $2$\,s of history and predict its position over a $6$\,s horizon with a constant velocity-and-heading baseline, which is the standard physics model of the nuScenes prediction
challenge \cite{phan2020covernet}. We set the nonconformity score to the final displacement error in meters,
\[
R^{(i)} = \big\| \hat{s}_T^{(i)} - s_T^{(i)} \big\|_2,
\]
the distance between the predicted and true position at horizon $T=6$\,s. The prediction set is now a ball of radius $\bar{\alpha}$ around the predicted endpoint, so the analogue of prediction-set size is the calibrated \emph{radius} in meters. Unlike the discrete label sets of Sections~\ref{subsec: image_classification}--\ref{subsec:language_models}, there
is no finite label space to exhaust. The additive DRO radius simply grows the ball, so the prediction set is never vacuous.

\textbf{Experimental setup.}
We score $49787$ agent trajectories from the nuScenes prediction challenge with the
constant-velocity-and-heading model. The final displacement error has a median of $9.4$\,m, a mean of $11.5$\,m, and
a $90$th percentile of $23.8$\,m. In each of the $1000$ trials, we draw $K=1000$ trajectories for calibration and use the remaining $n_{\text{eval}}=48787$ for evaluation, with $\delta=0.1$ and $\beta=0.1$. We compare Split CP, the three calibration-conditional CP variants, and $W_\infty$-DRO at additive radii $r_K\in\{1,3\}$\,m.

\textbf{Results.}
Figure~\ref{fig:nuscenes_coverage} reports the mean empirical coverage across the trials,
Figure~\ref{fig:nuscenes_two_level}  the two-level guarantee probability, and
Table~\ref{tab:nuscenes} the calibrated radius. The pattern matches the classification and language-model experiments: every method attains marginal coverage at or above the nominal $0.9$ but Split CP, which controls only marginal coverage, satisfies the two-level criterion with rate only $0.554$. The calibration-conditional methods (Lemma~\ref{lem:2}, Lemma~\ref{lem:4}) satisfy it
in essentially all trials with radii of $26$--$27$\,m, while Lemma~\ref{lem:3} tracks the
nominal level with a $24.8$\,m ball, and DRO reaches it once
$r\ge 3$\,m. The key difference from the discrete settings is the absence of vacuity. As the DRO radius grows from $1$ to $3$\,m the calibrated ball grows from $24.9$ to $26.9$\,m and never degenerates. 
%Here the additive Wasserstein radius is the natural object, a safety margin in meters.
%  rather than the blunt instrument it becomes on a bounded score.

Figure~\ref{fig:nuscenes_tube} shows the physical scale of the calibrated margin for a single turning agent. The constant velocity-and-heading prediction extrapolates straight and misses the turn, while the calibrated tube (per-step radius growing from $\approx 1$\,m at $0.5$\,s to $\approx 27$\,m at $6$\,s) still contains the true trajectory.
%, i.e., coverage expressed as a physical margin.

\begin{figure}[htb]
    \centering
    \includegraphics[width=\linewidth]{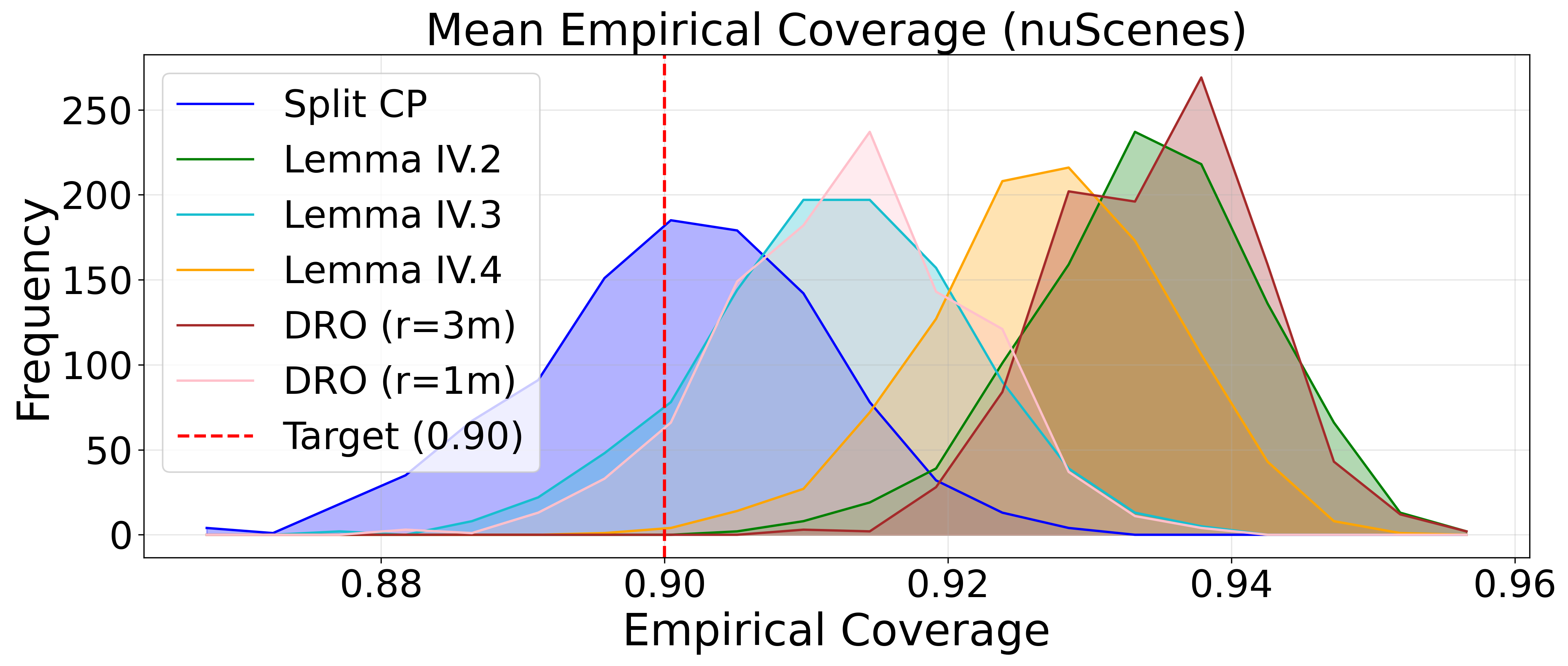}
    \caption{Mean empirical coverage across $1000$ random calibration/test
    splits of the nuScenes prediction set ($\delta=0.1$). All methods achieve
    marginal coverage at or above the nominal $0.90$; Split CP concentrates on the target while the calibration-conditional and DRO variants concentrate above it.}
    \label{fig:nuscenes_coverage}
\end{figure}

\begin{figure}[htb]
    \centering
    \includegraphics[width=\linewidth]{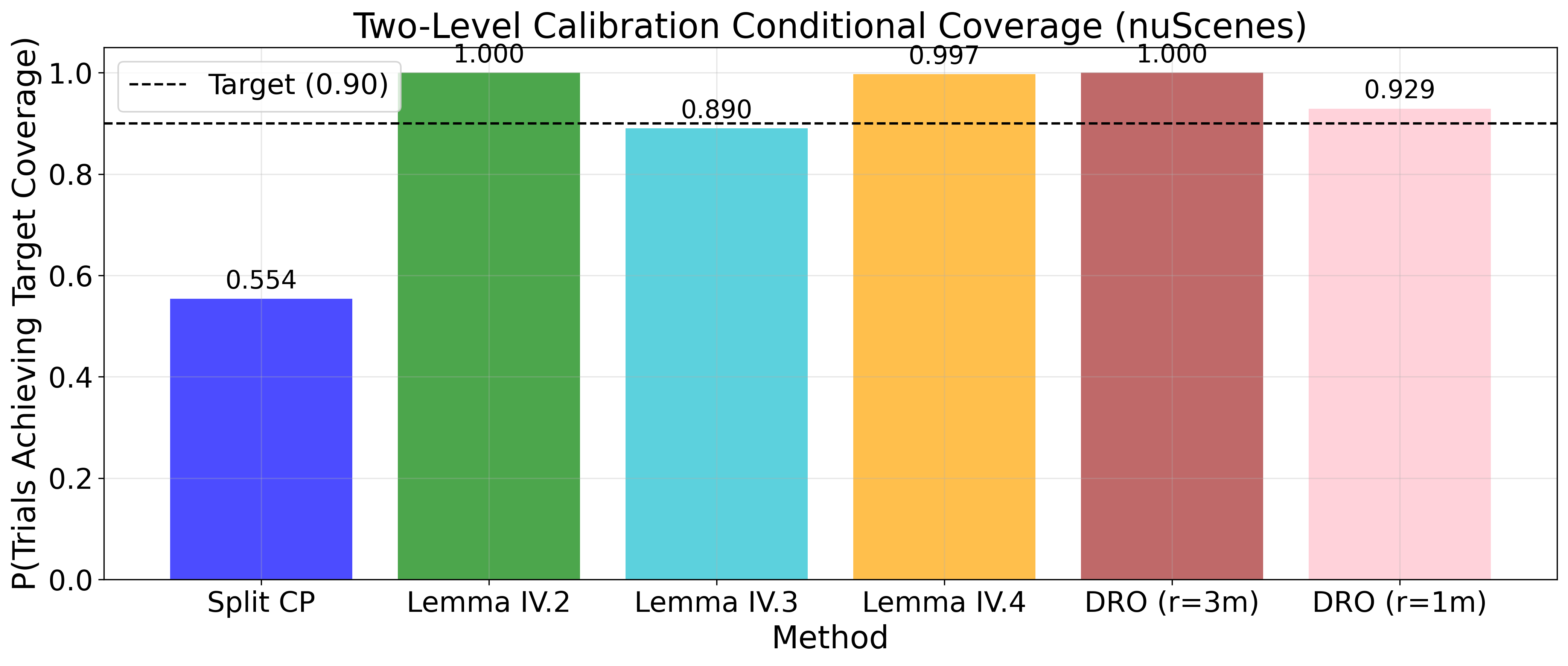}
    \caption{Two-level calibration-conditional coverage on nuScenes trajectory
    prediction (i.i.d., $\delta=\beta=0.1$, $1000$ trials). Plotted is the
    fraction of trials whose calibration set attains conditional coverage at least
    $1-\delta$, which estimates $\mathbb{P}^K$ in \eqref{eq:cal_cond_goal} and must reach $1-\beta$ for the guarantee to hold. Split CP reaches only $0.554$, Lemma~\ref{lem:2} and Lemma~\ref{lem:4} reach $1.000$ and $0.997$, Lemma~\ref{lem:3} reaches $0.890$, and DRO reaches the level once $r\ge 3$\,m.}
    \label{fig:nuscenes_two_level}
\end{figure}

\begin{table}[htb]
    \centering
    \small
    \caption{nuScenes trajectory prediction ($K=1000$, $\delta=0.1$, $\beta=0.1$,
    $1000$ trials, i.i.d.). ``Set size'' is the calibrated ball radius in meters
    (mean $\pm$ std). Split CP attains marginal but not two-level coverage; the
    calibration-conditional methods attain the two-level guarantee. Unlike the
    classification settings, the DRO radius grows smoothly with $r$ and is never
    vacuous, since the prediction region is a continuous ball rather than a
    discrete label set.}
    \label{tab:nuscenes}    
    \setlength{\tabcolsep}{4pt}
    \resizebox{\linewidth}{!}{\begin{tabular}{lccc}
    \toprule
    Method & Mean coverage & $\mathbb{P}^K(\text{cov}\geq 1{-}\delta)$ $\uparrow$ & Set size (m) $\downarrow$ \\
    \midrule
    Split CP       & $0.901$ & $0.554$ & $24.0 \pm 0.7$ \\
    Lemma~\ref{lem:2}     & $0.934$ & $1.000$ & $26.8 \pm 0.8$ \\
    Lemma~\ref{lem:3}     & $0.912$ & $0.890$ & $24.8 \pm 0.8$ \\
    Lemma~\ref{lem:4}     & $0.927$ & $0.997$ & $26.2 \pm 0.8$ \\
    DRO ($r_K=3$\,m) & $0.934$ & $1.000$ & $26.9 \pm 0.7$ \\
    DRO ($r_K=1$\,m) & $0.913$ & $0.929$ & $24.9 \pm 0.7$ \\
    \bottomrule
    \end{tabular}}
    \vspace{-2ex}
\end{table}

\begin{figure}[htb]
    \centering
    \includegraphics[width=\linewidth]{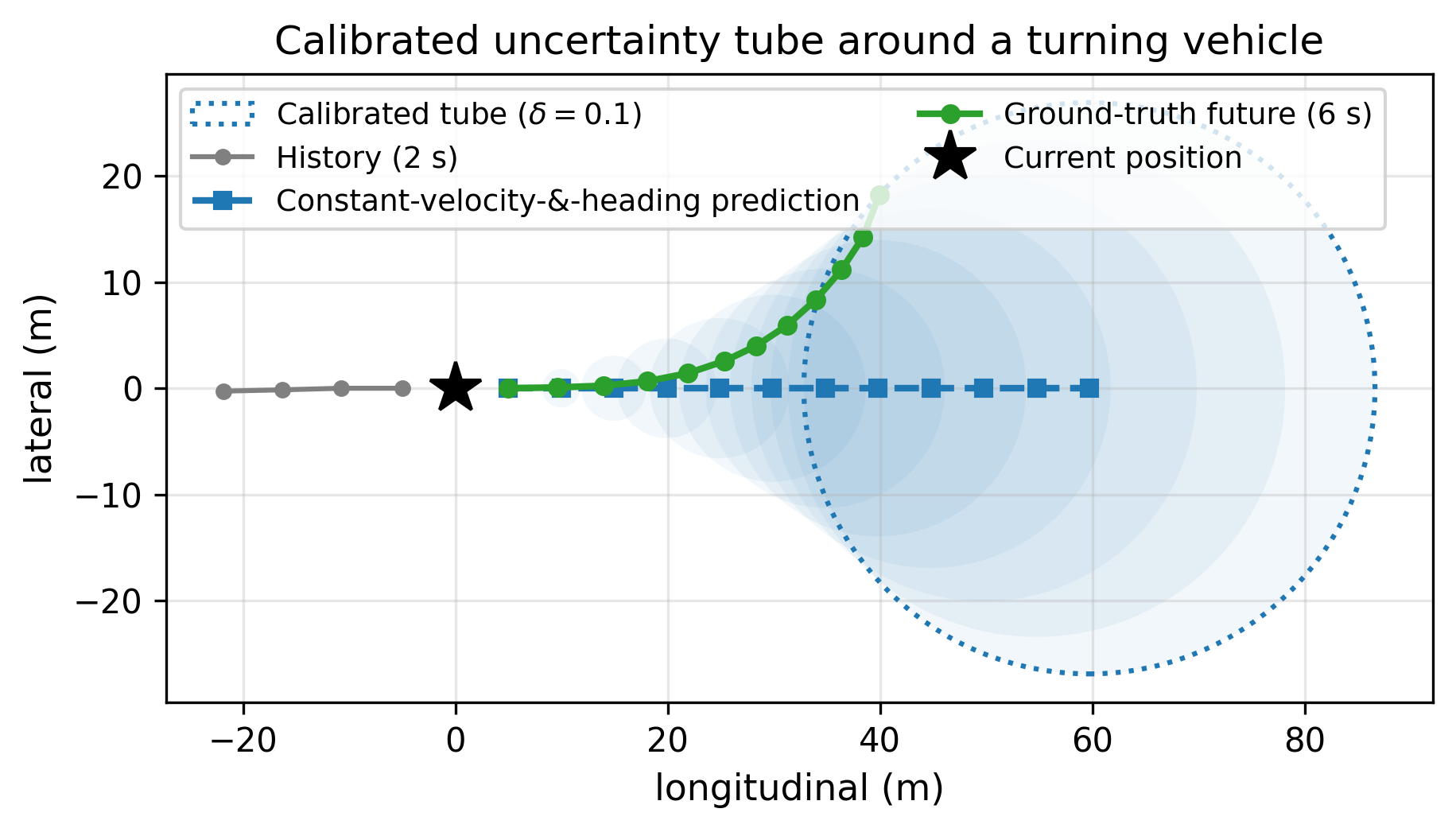}
    \caption{Calibrated uncertainty tube for a turning agent. The
    constant velocity-and-heading prediction (blue) extrapolates straight and misses the turn, while the ground-truth future (green) curves away. The per-step calibrated tube (radius growing from $\approx 1$\,m at $0.5$\,s to
    $\approx 27$\,m at $6$\,s) contains the ground-truth. The radii are calibrated separately at
    each time step at level $1-\delta$, so the figure shows the physical scale of a
    calibrated margin rather than a guarantee over the horizon.}
    \label{fig:nuscenes_tube}
    \vspace{-2ex}
\end{figure}

\begin{table}[t]
    \centering
    \small
     \caption{Maneuver shift on nuScenes ($K=1000$,
    $\delta=\beta=0.1$, $1000$ trials) with calibration on straight-driving agents and testing on turning agents. As in Table~\ref{tab:imagenet_c_shift}, each shift model has a budget $\eta$ in meters. ``Mean coverage'' is the average coverage over trials;
    $\mathbb{P}^K(\text{cov}\geq 1{-}\delta)$ is the calibration-conditional success rate
    $\mathbb{P}^K(\mathbb{P}_{\mathrm{test}}\{R^{(0)}\le\bar\alpha\}\ge1-\delta)$;
    ``Set size'' is the calibrated ball radius in meters. The
    Split CP margin under-covers the shifted population ($0.669$). The
    shift-aware estimators recover coverage, with the marginal method
    (Theorem~\ref{thm:lp_robust_cp}) attaining marginal but not two-level coverage and Theorem~\ref{thm:pac_levy} attaining both; its DRO counterpart (Theorem~\ref{thm:lp_dro_shift}) attains both with a tighter ball.
    }
    \label{tab:nuscenes_shift}
    \setlength{\tabcolsep}{4pt}
    \resizebox{\linewidth}{!}{\begin{tabular}{llccc}
    \toprule
    Method & $(\eta,\rho)$ & Mean coverage & $\mathbb{P}^K(\text{cov}\geq 1{-}\delta)$ $\uparrow$ & Set size (m) $\downarrow$ \\
    \midrule
    Split CP                       & -- & $0.669$ & $0.000$ & $20.3 \pm 0.7$ \\
    \midrule
    \multicolumn{5}{l}{\emph{$W_\infty$ shift model} (Assumption~\ref{assump: test_shift_wasserstein})} \\
    Lemma~\ref{lem:cp_shift}       & $(8,-)$ & $0.918$ & $0.966$ & $31.1 \pm 0.8$ \\
    Lemma~\ref{lem:dro_shift}      & $(8,-)$ & $0.923$ & $0.991$ & $31.5 \pm 0.8$ \\
    \midrule
    \multicolumn{5}{l}{\emph{L\'evy--Prokhorov shift model} (Assumption~\ref{assump:test_shift_LP})} \\
    Theorem~\ref{thm:lp_robust_cp} & $(5,0.05)$ & $0.905$ & $0.657$ & $30.2 \pm 0.9$ \\
    Theorem~\ref{thm:pac_levy}     & $(5,0.05)$ & $0.969$ & $1.000$ & $36.7 \pm 1.4$ \\
    Theorem~\ref{thm:lp_dro_shift} & $(5,0.05)$ & $0.941$ & $0.999$ & $33.1 \pm 1.3$ \\
    \bottomrule
    \end{tabular}}
\end{table}

\textbf{Distribution shift.}
We additionally probe two covariate shifts, calibrating on one subpopulation and testing on another. A \emph{geographic} shift, where we calibrate on Boston and test on Singapore, barely moves the score distribution (the $90$th-percentile error is $23.9$\,m versus $23.7$\,m) because constant-velocity error depends on agent dynamics rather than city. Thus, we consider a \emph{maneuver} shift, where we calibrate on straight driving and test on turning. Turning roughly doubles the prediction error, raising the $90$th-percentile from $20.2$ to $29.7$\,m. The i.i.d. Split CP margin under-covers the turning population at only $0.669$. Table~\ref{tab:nuscenes_shift} shows that the shift-aware estimators of Section~\ref{subsec:dist_shift} recover coverage, with the same marginal-versus-two-level distinction as the ImageNet-C study (Table~\ref{tab:imagenet_c_shift}): the marginal estimator (Theorem~\ref{thm:lp_robust_cp}) achieves marginal coverage ($0.905$) but satisfies the two-level guarantee with rate only $0.657$, whereas Theorem~\ref{thm:pac_levy} attains both ($0.969$, $1.000$). Because the score is continuous, the correction enlarges the safety ball
smoothly (from $20$ to $30$--$37$\,m) rather than collapsing to a vacuous label set seen in classification. The budget behaves differently here. On the bounded ImageNet-C score, set size grows far more slowly in $\rho$ than in $\eta$
(Figure~\ref{fig:imagenet_c_rrho}), whereas on this unbounded distance score every budget that achieves marginal coverage yields essentially the same ball, $30$--$31$\,m, so the two parameters are interchangeable. This is the density scaling of Section~\ref{subsec:stat_props} in the shifted setting: $\rho$ shifts the quantile level and $\eta$ shifts its value, and the score density at the threshold converts one into the other.

\subsection{Choosing an Estimator}
\label{subsec:takeaways}

This section gives guidance on which estimator to use. Although the three applications (Sections~\ref{subsec: image_classification}, \ref{subsec:language_models} and~\ref{subsec:autonomous_driving}) differ in label space and in whether the score is bounded, they rank the CP estimators identically by both set size and two-level rate, so
the same guidance applies to all three.

If only marginal coverage is required, split CP is the standard choice and gives the smallest sets everywhere. However, it does not control coverage for the particular calibration set drawn, and its two-level rate is far below $0.9$ in the three studies.
%, so about half of the calibration sets give coverage below $1-\delta$.}

Among the calibration-conditional corrections, Lemma~\ref{lem:3} gives the
smallest sets, because it solves the beta-function condition exactly, whereas Lemmas~\ref{lem:2} and~\ref{lem:4} relax it. Because Lemma~\ref{lem:3} carries no slack, its two-level rate tracks $1-\beta$ rather than exceeding it and the measured value can fall on either side. Lemmas~\ref{lem:2} and~\ref{lem:4} achieve a wide margin at the cost of larger prediction sets, and their level correction is explicit, so it can be allocated across constraints or time steps. Of the two, Lemma~\ref{lem:4} is preferable once $\delta$ falls below roughly $0.1$.

Without calibration-test distribution shift, DRO's certified radius gives the larger correction (Sections~\ref{sec:examples} and~\ref{sec:exp_setup_cp_cls}), so CP produces smaller prediction sets and is the preferable choice. DRO is nevertheless preferable in two cases: when $K$ lies below the non-vacuity thresholds, where CP returns no finite estimator at all, and when coverage must hold over an ambiguity set rather than for $\bbP$ alone. Whether the score is bounded also decides what DRO's additive correction costs: on the continuous nuScenes distance DRO matches the calibration-conditional sets ($26.9$ against $26.8$\,m at a two-level rate of $1.000$), whereas on the bounded MMLU score the same correction returns the full label set in a fraction $0.160$ of trials.

Every shift-aware estimator requires the shift budget as an input: $\eta$ for the $W_\infty$ corrections of Lemmas~\ref{lem:cp_shift} and~\ref{lem:dro_shift}, and $(\eta,\rho)$ for the LP estimators. None is assumption-free. When the shift is purely value-space, the
$W_\infty$ corrections are the tighter choice, as they do not pay for a level perturbation ($31.1$ against $36.7$\,m on nuScenes). Under the LP model, Theorem~\ref{thm:lp_robust_cp} achieves marginal coverage only ($0.905$, with a
two-level rate of $0.657$), whereas Theorem~\ref{thm:pac_levy} and Theorem~\ref{thm:lp_dro_shift} achieve both. The choice between those two again follows the score, the latter giving the tighter ball on the continuous nuScenes distance but degenerating on the bounded ImageNet-C score. How to split a given budget between $\eta$ and $\rho$ depends on the score as well: on a
bounded score $\rho$ inflates sets far more slowly than $\eta$
(Figure~\ref{fig:imagenet_c_rrho}), whereas on an unbounded distance score the
two are interchangeable, so there $\eta$ may be set directly from the
anticipated perturbation.

%
% \marginJC{B/c of the large number of experiments and tables/figures (which I like!), I wonder if it would be helpful to say/show what is better in each one (e.g., larger coverage is better, smaller set size is better, etc.) so that it is easy to compare approaches.}
%

% \begin{figure}[t]
%     \centering
%     \includegraphics[width = 0.44\textwidth]{fig/dro_calibrate_radius/cp_calibrated_dro_delta0.5_beta0.1_use_approximationTrue.png}
%     \caption{Comparison of CP and DRO quantile estimates with approximated calibrated radius in \eqref{eq: uniform_approximate_radius}.}
%     \label{fig:match_mean_approx}
% \end{figure}

% \section{Potential Directions}

% Some potential directions are presented as follows: 

% \begin{itemize}
%     \item Safe robot motion planning with trajectory data (e.g., \cite{lindemann2023safe})
%     \item Robust model-based RL / certificate function (e.g., CLF, CBF) learning under model uncertainty (e.g., \cite{long2023dro_lf, Robey_2020_learn_CBF_demonstration, Dawson_2023_TRO})
%     \item Mixed integer QP formulation for CLF-CBF based approaches under uncertainty (in states, dynamics, barreir functions and its gradients). \cite{zhao2024conformal_cc_opti, long2024sensorbased_dro}

% \end{itemize}

%% file: tex/Conclusion.tex
\section{Conclusions}
% \& Future Work
\label{sec:conclusion}

This article developed a unified probabilistic perspective on conformal prediction (CP) and Wasserstein distributionally robust optimization (DRO), viewing both as procedures that turn finite calibration data into a quantile threshold with finite-sample coverage. The two correct the empirical quantile along different coordinates: CP raises the target quantile level, whereas DRO shifts the quantile value through an ambiguity radius. Making the correspondence precise exposes an asymmetry in what each method certifies. CP's level correction is closed-form and distribution-free, whereas the radius that certifies DRO's coverage depends on a density lower bound that is typically unknown; in return, DRO certifies coverage uniformly over the ambiguity set rather than for the true distribution alone.

In the scalar case, the comparison is exact. The two estimators share the same first-order asymptotic variance and differ only in a deterministic offset. With certified radius, that offset is provably larger for DRO and is governed by a uniform lower bound on the density, whereas CP's is governed by the density at the target quantile.

In the experiments, calibration-conditional CP and DRO attain the two-level guarantee that the standard split-CP baseline fails to meet. Among the methods that meet it, CP produces the smallest prediction sets, while DRO's additive correction is well suited to a continuous score but not to a bounded score, where it can drive the prediction set to vacuity. The same distinction applies to the shift-aware estimators under both shift models. These behaviors are consistent across ImageNet classification, MMLU question answering, and nuScenes trajectory prediction.
% , clean and corrupted

Two open problems follow. The first is the DRO radius. The value that certifies coverage depends on unknown properties of the distribution, is provably the looser correction, and can saturate a bounded score. Data-driven and adaptive radii can help, but they require more than the $K$ calibration scores, such as a model for the score distribution or repeated problem instances. The second is the shift budget. Every shift-aware estimator requires it as an input, so none is assumption-free. The open problem is to estimate the budget from data and to account for the estimation error in the guarantee. Beyond these, a further open direction is to propagate the two-level guarantee from coverage into decision-making, e.g., in chance-constrained optimization, model predictive control, and planning, where a calibrated threshold would provide a closed-loop guarantee. Another direction is related to online calibration, where data arrive as a non-stationary stream and the threshold must be updated continually to preserve coverage under drift, as in autonomous systems operating over long horizons.